\documentclass[11pt]{article}
\usepackage{graphicx,amsthm,graphicx,fancyhdr,mathrsfs}
\usepackage{amsfonts}
\usepackage{amsmath}
\usepackage{amssymb}
\usepackage{url}
 \usepackage[colorlinks=true,citecolor=blue]{hyperref}
\usepackage{fancyhdr}
\usepackage{indentfirst}
\usepackage{enumitem}
\usepackage{amsthm}
\usepackage[dvipsnames]{xcolor}
\usepackage{comment}
\usepackage{dsfont}
\usepackage{natbib}
\usepackage[misc]{ifsym}
\usepackage[toc,page]{appendix}
\usepackage{multirow}
\usepackage{float}
\usepackage{bm}
\usepackage{booktabs}
\usepackage{todonotes}
\usepackage[font=scriptsize]{caption} 
\usepackage{authblk}

\def\d{\,\mathrm{d}}

\newcommand{\VaR}{\mathrm{VaR}}

\newcommand{\ES}{\mathrm{ES}}

\newcommand{\LR}{\mathrm{LR}}

\newcommand{\E}{\mathbb{E}}
\newcommand{\R}{\mathbb{R}}

\renewcommand{\P}{\mathbb{P}}

\renewcommand{\geq}{\geqslant}
\renewcommand{\leq}{\leqslant}

\theoremstyle{plain}
\newtheorem{theorem}{Theorem}
\newtheorem{corollary}{Corollary}
\newtheorem{lemma}{Lemma}
\newtheorem{proposition}{Proposition}
\theoremstyle{definition}
\newtheorem{definition}{Definition}

\newtheorem{assumption}{Assumption}

\newtheorem{remark}{Remark}
\newtheorem{problem}{Problem}
\DeclareMathOperator*{\argmin}{arg\,min}

\usepackage{tikz}

\renewcommand{\cite}{\citet}

\makeatletter
\DeclareRobustCommand{\bsquare}{%
	\mathop{\vphantom{\sum}\mathpalette\bigstar@\relax}\slimits@
}

\newcommand{\bigstar@}[2]{%
	\vcenter{%
		\sbox\z@{$#1\sum$}%
		\hbox{\resizebox{.9\dimexpr\ht\z@+\dp\z@}{!}{$\m@th\dsquare$}}%
	}%
}
\makeatother

\usepackage[onehalfspacing]{setspace}

\newcommand{\dsquare}{\mathop{  \square} \displaylimits}

\title{Dynamic reinsurance design with heterogeneous beliefs\\ under the mean-variance framework}

\author[a]{Junyi Guo}
\author[b]{Xia Han}
\author[c]{Hao Wang\thanks{Corresponding author.

\ \ \text{E-mail addresses: jyguo@nankai.edu.cn (J. Guo); xiahan@nankai.edu.cn (X. Han);}

\ \ \text{hao.wang@mail.nankai.edu.cn (H. Wang)}}}
\affil [a] {School of Mathematical Sciences and LPMC, Nankai University, Tianjin, 300071, China}
\affil [b] {School of Mathematical Sciences, LPMC and AAIS, Nankai University, Tianjin, 300071, China}
\affil [c] {School of Mathematical Sciences, Nankai University, Tianjin, 300071, China}

\date{}

\begin{document}

\maketitle

%\linenumbers

\begin{abstract}
This paper investigates the dynamic reinsurance design problem under the mean-variance criterion, incorporating heterogeneous beliefs between the insurer and the reinsurer, and introducing an incentive compatibility constraint to address moral hazard. The insurer's surplus process is modeled using the classical Cram\'er-Lundberg risk model, with the option to invest in a risk-free asset. To solve the extended Hamilton-Jacobi-Bellman (HJB) system, we apply the partitioned domain optimization technique, transforming the infinite-dimensional optimization problem into a finite-dimensional one determined by several key parameters. The resulting optimal reinsurance contracts are more complex than the standard proportional and excess-of-loss contracts commonly studied in the mean-variance literature with homogeneous beliefs. By further assuming specific forms of belief heterogeneity, we derive the parametric solutions and obtain a clear optimal equilibrium solution. Finally, we compare our results with models where the insurer and reinsurer share identical beliefs or where the incentive compatibility constraint is relaxed. Numerical examples are provided to illustrate the impacts of belief heterogeneity and the incentive compatibility constraint on optimal reinsurance strategies.
\bigskip

\\


\bigskip

%\noindent{{\bf MCS Subject Classification (2020):} }
\end{abstract}

\section{Introduction}
%\label{eq}
\subsection{Background}
Over the past decade, optimal reinsurance has become a central focus in actuarial science, owing to its practical significance. Insurers adjust their retention levels, either continuously or in a {single period}, based on market dynamics and internal factors, to manage risk within their portfolios. 
An appropriate optimization criterion is crucial for deriving the optimal reinsurance strategy. The existing literature can be broadly categorized into two main streams: one that focuses on maximizing expected utility (EU), and another that emphasizes risk minimization. Both approaches have stimulated extensive research. This paper is rooted in the EU maximization framework, with a particular focus on the mean-variance criterion.

Since the foundational work of \cite{M52}, the mean-variance criterion has become a cornerstone in mathematical finance, with stochastic control theory playing a central role in dynamic optimal (re)insurance and investment problems within this framework. It is well recognized that dynamic  problems under the mean-variance criterion suffer from time-inconsistency, as the Bellman optimality principle fails due to the loss of the iterated-expectation property. In particular, the optimal control strategy derived from a fixed initial condition may no longer be optimal in subsequent periods.

To address the issue of time-inconsistency, two primary approaches have been proposed in the literature. The first approach utilizes pre-committed strategies, where decision-makers focus exclusively on achieving optimality at the current time, without regard to future optimality. This method has been extensively studied in the context of mean-variance portfolio selection and (re)insurance optimization, with  contributions from \cite{L00}, \cite{ZL00}, \cite{B05}, \cite{BG13}, \cite{SZ14}, \cite{GXL16}, \cite{DJKX21}, among others. While pre-committed strategies offer valuable theoretical insights and can be economically meaningful in specific situations, they do not resolve the time-inconsistency problem. This is because they disregard the dynamic nature of decision-making over time, failing to account for how optimal decisions might evolve as time progresses.

In contrast, the second approach, which has its roots in the seminal work of \cite{S56}, seeks to address time-inconsistency by framing the problem within the context of non-cooperative game theory. This approach treats decision-making as a dynamic game, where a subgame perfect Nash equilibrium is sought at each stage of the process. By modeling the decision-making process as a game against future decision makers, this approach ensures that the resulting equilibrium strategies are time-consistent. The game-theoretic approach to time-inconsistency has  been further developed in various directions. For example, \cite{BM10} introduces a game-theoretic framework and derived an extended Hamilton-Jacobi-Bellman (HJB) equation using a verification theorem for a
fairly general objective functional. Expanding on this, \cite{BMZ14} considers the case where the insurer's risk aversion is inversely proportional to his current wealth, deriving a corresponding time-consistent strategy. Additionally, \cite{ZLG16} analyzes the equilibrium investment-reinsurance strategy for an ambiguity-averse insurer concerned with model uncertainty, while \cite{YLH21} designs a robust reinsurance contract within the context of a Stackelberg differential game, considering insurers and reinsurers with mean-variance preferences. For further discussion on (re)insurance problems under the mean-variance criterion, we refer readers to \cite{LLX16}, \cite{CS19}, \cite{LY21}, \cite{CLLL21}, \cite{YHLY23} and the reference therein.

The studies mentioned above primarily focus on continuous models. In contrast, more recent literature has also explored static models within the mean-variance framework, employing different methodologies to address the criterion. \cite{CZ17} applies the mean-variance criterion to study an optimal reinsurance problem, where the strategy is priced based on the mean and variance of the indemnity function. \cite{BJ22} examines the optimal insurance design from the perspective of the insured, considering the possibility that the insurer may default on his promised indemnity function. \cite{LY21} generalizes the mean-variance framework to a broader context, specifically investigating a one-period, mean-variance Stackelberg game. Furthermore, \cite{LJZ23} studies an optimal insurance design problem under the mean-variance criterion, incorporating the local gain-loss utility of the net payoff from insurance, a concept known as narrow framing. Lastly, \cite{CJZ23} explores an optimal insurance problem within the mean-variance framework, focusing on the scenario where the insured and insurer hold heterogeneous beliefs about the loss distribution.

\subsection{Research Motivation and Problem Statement}
Except the contributions of \cite{CJZ23}, the studies discussed above generally assume that both parties in an insurance contract share identical probabilistic beliefs concerning the loss distribution. However, this assumption has faced substantial criticism within the field of insurance economics. It is well-documented that information asymmetry is pervasive in insurance markets, where the insurer and the insured often have access to different private information about the underlying random loss. Given that both parties make decisions based on their own perceptions of uncertainty, it is therefore natural for them to hold heterogeneous beliefs about the loss distribution. Consequently, a growing body of literature has examined optimal (re)insurance design particular in static settings under belief heterogeneity, employing a variety of criteria. Notable contributions  include those by \cite{B15}, \cite{BG21}, and \cite{GZC24}, who focus on minimizing specific risk measures,  as well as by \cite{BG19}, \cite{G17},  \cite{CZ17} and \cite{CZ20}, who consider the  utility  function.

Building on the insights from the studies discussed above, to the best of our knowledge, the design of mean-variance insurance within dynamic reinsurance models in the context of belief heterogeneity, remains an unexplored area, which in turn motivates us to undertake the present research. In addition, we incorporate an incentive compatibility condition which requires the losses assumed by both insurer and reinsurer to increase with the actual losses to preclude moral hazard. Specifically, we assume that the insurer's surplus process follows the classical Cramér-Lundberg risk model, where the insurer can purchase reinsurance from a reinsurer.\footnote{We adopt the convention of referring to the insurer with the pronoun `he' and the reinsurer with the pronoun `she'.} The reinsurance premium is determined based on the expected value principle. Furthermore, the insurer is permitted to invest his surplus in a risk-free asset. Our objective is to seek the optimal time-consistent strategy for the mean-variance problem within a game theoretic framework under the assumptions that both the insurer and the reinsurer have heterogeneous beliefs about the claims, and that the insurer selects reinsurance strategies from a range of contracts that satisfy the incentive compatibility constraint.

\subsection{Literature Review}

The incentive compatibility constraint has long been a central concern in contract theory, particularly within static reinsurance frameworks; see, e.g., \cite{A15} and the review by \cite{CC20}, along with the references mentioned earlier.  However, in the context of dynamic reinsurance models, only a limited number of studies have directly addressed this issue, each under different objectives. Notably, \cite{TWWZ20} and \cite{JXZ24} focus on minimizing the ruin probability. Both \cite{MWZZ22} and \cite{MWZ24} incorporate belief heterogeneity into their models; the former investigates the maximization of expected exponential utility, while the latter adopts the Lundberg exponent as the objective function. In addition, \cite{XZZ19} and \cite{X23} examine optimal insurance problems in which the insured seeks to maximize rank-dependent utility preferences.   

In studies related to belief heterogeneity, the likelihood ratio (LR) has been widely used to model the differing beliefs between the insurer and the reinsurer. The majority of related studies have addressed static problems. Moreover, LR is often assumed to be smooth or monotonic. (see, e.g.,  \cite{G13}, \cite{MWZZ22} and \cite{GJR22}). Specifically, a decreasing LR indicates that the reinsurer is more optimistic about the size of the underlying loss compared to the insurer, while an increasing LR  suggests that the insurer holds a more optimistic view than the reinsurer. However, such a framework does not capture the full spectrum of belief heterogeneity between the insurer and the reinsurer.   To address this limitation, some works partially relax the assumptions of monotonicity and continuity of the LR (see, e.g,  \cite{CZ20} and  \cite{MWZ24}).

\subsection{Main Contributions}
This paper aims to investigate the dynamic reinsurance design problem under the mean-variance criterion, explicitly incorporating heterogeneous beliefs and an incentive compatibility constraint, thereby introducing significant additional challenges. In existing studies of dynamic mean-variance framework, either the reinsurance strategy is directly assumed to be in the form of quota share or excess-of-loss reinsurance, or it is shown that these types of contracts are optimal under certain premium principles. In such cases, incentive compatibility is naturally satisfied. However, when belief heterogeneity is introduced, the situation becomes significantly more complex, and incentive compatibility no longer holds by default. This issue is carefully examined in our work.  We employ the partitioned domain optimization technique {which is  well-established in addressing} the static problem (see \cite{CYW13}, \cite{BJ22}, \cite{GJR22}, and \cite{CJZ23}), and extend this methodology to solve the associated extended HJB system for our  dynamic mean-variance problem. This approach yields a parametric form of the solution, thereby enabling the transformation of an infinite-dimensional optimization problem into a finite-dimensional one.
 
In our paper, we adopt the definition of the LR function from \cite{EG79}, further extending it to a more general class of functions, which is a significant departure from the existing literature. This relaxation enables us to capture a broader spectrum of belief heterogeneity and accommodate the possibility of a point mass at zero in the insurable loss distribution — a feature commonly observed in insurance practice, as highlighted by \cite{S68}. Additionally, our framework encompasses two widely used risk measures, Value at Risk (VaR) and Expected Shortfall (ES), which can serve as potential premiums.

In contrast to the existing  dynamic literature mentioned above, we further explore  the  relationship between strategy and time. The relaxation of smoothness in the LR  function and distribution function of claims introduces  additional complexity to this relationship. Nonetheless, we show that, under mild conditions, the strategy remains continuous with respect to time. Additionally, for specific forms of belief heterogeneity, we provide a parametric solution and solve for the optimal parameters, leading to a clear equilibrium outcome. In conclusion, our work significantly extends the existing literature on mean-variance frameworks.

  %In contrast to \cite{CJZ23}, which considers optimal reinsurance under a  static framework,  our methodology in the dynamic setting differs.  We also allow the insurer to continuously invest in the risk-free market. Moreover, compared to  existing  dynamic literature (see, e.g., \cite{MWZZ22} and \cite{MWZ24}), the relaxation of smoothness in LR makes the relationship between strategy and time more complex. However, we show that, under mild conditions, the strategy is still continuous with respect to time. Furthermore, for specific forms of belief heterogeneity, we provide a parametric solution and solve for the optimal parameters, leading to a clear equilibrium outcome.    In conclusion, our work significantly extends the existing literature on mean-variance frameworks, both static and dynamic.

\subsection{Organization of the Paper}
The remainder of the paper is organized as follows. Section \ref{sec:2} formulates the primary problem under investigation. In Section \ref{sec:3}, we first define the equilibrium strategy, then establish its existence and uniqueness, followed by the derivation of its explicit form. Section \ref{sec:4} examines three specific forms of belief heterogeneity, where the premium is calculated using a distortion risk measure. Section \ref{sec:5} presents numerical examples and illustrates the impacts of belief heterogeneity and the incentive compatibility constraint on the model. %We also compare our results with those obtained under the assumption that the insurer and reinsurer share identical beliefs, as well as with a scenario in which the incentive compatibility constraint is relaxed. 
Finally, Section \ref{sec:6} concludes the paper. The proofs of the main results are provided in the appendix.
\section{Formulation of problem}\label{sec:2}
Let $(\Omega,\mathcal F,\mathbf F=\{\mathcal F_t\}_{t\in[0,T]},\mathbb P)$  be a filtered probability space satisfying the usual conditions of
completeness and right continuity, and   $T>0$ be a finite time horizon. We assume that an insurer's surplus process is modeled by the Cramér-Lundberg risk model
\begin{align*}
	\widetilde X(t) = u + ct - \sum\limits_{i=1}^{N(t)}Y_i\ ,
\end{align*}  
where $\widetilde X(0)=u$ is the initial surplus, $N(t)$ is a Poisson process which counts the number of claims occurring in the time horizon $[0,t]$ with intensity $\beta>0$, $\{Y_i\}_{i=1}^{\infty}$ is a sequence of independent and identically distributed positive random variables   which denote each claim size and is independent of $N(t)$, and $c>0$ is the premium rate.    Let $Y$ be the generic random variable which has the same distribution as $Y_i$.  We  assume that $Y$ has finite first-order and second-order moments, i.e. $\mathbb E[Y]<\infty$ and {$\mathbb E[Y^2]<\infty$}; without it, the mean-variance
criterion cannot be applied.  Without loss of generality, we assume $\beta = 1$ and $c>\mathbb E [Y]$, which implies that the insurer's premium income exceeds the expected claims payout per unit of time. { This assumption is standard in classical ruin theory and ensures that  the surplus process $\{ \tilde X(t) \}_{t \geq 0}$ does not fall below zero with probability one.}

In this paper, we consider the insurer who can purchase a reinsurance contract to reduce his risk exposure, that is, the insurer can transfer the claims to a reinsurer under a reinsurance arrangement via a continuously payable reinsurance premium. {Let $I(t,y):[0,T]\times \mathbb R_+\to \mathbb R_+$ denote the indemnity function, which satisfies $I(t, y) \leqslant y$ for all $(t, y)$ and represents the amount ceded to the reinsurer when a claim size $y$ occurs at time $t$. 
%Since part of the  subsequent analysis is conducted at a fixed time $t$, we also define the time-specific  indemnity function  $I_t : \mathbb R_+ \rightarrow \mathbb R_+$, with $I_t(y):=I(t,y)$ and $I_t(y)\leqslant y$ for all $y\geq0$. 
The  retained loss for  the insurer is then given by $R(t,y):=y-I(t,y)$.} In addition to reinsurance, the insurer also invests his surplus into a risk-free asset with a constant interest rate $r>0$. After considering reinsurance and investment, the surplus process of the insurer at time $t$ under the reinsurance strategy $I(t,y)$, denoted by $X(t)$, can be represented as 
\begin{align}\label{dynamic}
	\d X(t)=(c-c(I(t,\cdot))+rX(t-))\d t-\d\sum\limits_{i=1}^{N(t)}(Y_i-I({\tau_i},Y_i)),~~~~t\in[0,T],
\end{align}
where $c(I(t,\cdot))$ represents the reinsurance premium rate based on the insurer's indemnity function  $I(t,y)$, and $\tau_i$ denotes the time at which  the $i$th claim occurs.

To avoid potential ex post moral hazard, where the insurer might be tempted to manipulate losses,  we follow the literature (see, for example, \cite{HMS83} and \cite{CD03})  and impose the incentive compatibility condition  (also known as the no-sabotage condition) on the indemnity function. Specifically, we restrict our consideration to indemnity functions from the class defined below:
\begin{align*}
	\mathcal C := \{f:  \R_+\to\R_+   \mid  f(0)=0, 0\leqslant f(y)-f(x)\leqslant y-x, \forall x\leqslant y\}.
\end{align*}
 For any $f \in \mathcal C$, both $f(x)$ and $x-f(x)$ are increasing in $x$,  and every $f\in\mathcal C$ is 1-Lipschitz continuous.  
Since Lipschitz-continuous functions are absolutely continuous,  they are almost everywhere differentiable. It  follows directly  that  $f$ can be written as the integral of its derivatives, which is essentially bounded by   $1$.  Consequently,  we have \begin{equation}\label{eq:marginal}
\mathcal C=\left\{f:  \R_+\to\R_+  \mid  f(0)=0, f(x)=\int_0^x q(t) \d t, ~0 \leq q \leq 1,~ x\in\R_+\right\},
\end{equation}
where $q$ is called the  \emph{marginal indemnity function} \citep{A15,ZWTA16}.

The class $\mathcal{C}$ is rich enough and includes many commonly used indemnity functions, such as the excess-of-loss function $f(x)=(x-d)_{+}$ for some $d \geq 0$ and the quota-share function $f(x)=q x$ for some $q \in[0,1]$.  It is easy to check that for a fixed $t\in[0,T]$,  we have  $I(t,\cdot)\in\mathcal C\iff R(t,\cdot)\in\mathcal C$, and  thus there is no essential difference between considering $I$ and  $R$ in our model.  

Next, we define the admissible reinsurance strategies as follows.
\begin{definition}\label{def:1}
	A reinsurance strategy is a sequence of indemnity {functions} $\bm{I} = \{I(t,\cdot)\}_{t\in[0,T]}$ and is admissible if it is $\mathbf F$-predictable  and satisfies $I(t,\cdot)\in\mathcal C$ for all $t\in[0,T]$. We denote by $\mathcal I$ the set of all admissible strategy processes of the insurer.
\end{definition}

In our model, we explore a setting in which the insurer and reinsurer have heterogeneous beliefs, meaning their subjective assessments of the claim {$Y$}  differ. However, both parties are fully aware of each other's beliefs.  This can be mathematically characterized by  two distinct probability measures, $\mathbb P$ and $\mathbb Q$. Specifically, let $(\Omega_Y,\mathcal F_Y)$ be a measurable space, and let  $\mathbb P_1$ and $\mathbb Q_1$ be two probability measures on this space. The random variables $\{Y_i\}_{i=1,2,...}$ are independent and identically distributed random variables under $\mathbb P_1$ and $\mathbb Q_1$, respectively. Further, let $(\Omega_N,\mathcal F_N,\mathbb P_N)$ be another probability space, with   $N(t)$ denoting a Poisson process defined on this space. Then we can define the product measurable space $(\Omega_Y\times\Omega_N,\mathcal F_Y\times\mathcal F_N)$ as $(\Omega,\mathcal F)$, with two probability measures $\mathbb P = \mathbb P_1\times\mathbb P_N$ and $\mathbb Q = \mathbb Q_1\times\mathbb P_N$.  Let $\mathbb P$ denote the insurer's subjective probability, and $\mathbb Q$  the reinsurer's subjective probability. We assume  that {$Y$ also possesses finite first-order and second-order moments} under $\mathbb Q$. The cumulative distribution functions of claim $Y$ under $\mathbb P$ and $\mathbb Q$ are denoted by $F^{\mathbb P}$ and $F^{\mathbb Q}$, respectively.  The survival function of $Y$ under these two probabilities are given by $S^{\mathbb P}$ and $S^{\mathbb Q}$.

Suppose the reinsurer prices indemnity functions using the expected value premium principle based on her own subjective probability $\mathbb{Q}$. This method is widely adopted in the insurance literature due to its simplicity and economic relevance.  First, we use Poisson random measure $N(\d s,\d y)$ with intensity measure $\mu(\mathrm dy)=\d F^{\mathbb Q}(y)$ to represent the aggregate claims of reinsurer, and we have
\begin{align*}
	\sum\limits_{i=1}^{N(t)}I(\tau_i,Y_i)=\int_0^t\int_{\mathbb R_+}I(s,y)N(\d s,\d y),
\end{align*}
and 
\begin{align*}
	\mathbb E^{\mathbb Q}\left(\sum\limits_{i=1}^{N(t)}I(\tau_i,Y_i)\right)=\int_0^t\int_{\mathbb R_+}I(s,y)\mu(\mathrm d y)\d s.
\end{align*}
Then, the premium is given by:
\begin{align}\label{premium}
	c(I(t,\cdot))=(1+\theta)\int_{\mathbb R_+}{I(t,y)}\d F^{\mathbb Q}(y)=(1+\theta)\mathbb E^{\mathbb Q}[I(t,Y)].
\end{align}
where $\mathbb E^{\mathbb Q}$ denotes the expectation under probability $\mathbb Q$, and $\theta\geqslant 0$ is the safe loading factor.

Suppose that  the insurer uses mean-variance criterion under his belief $\mathbb P$ to determine reinsurance strategy, that is, the objective function of the  the insurer is given by 
\begin{align}\label{eq:J}
	J(t,x;\bm{I})=\mathbb E^{\mathbb P}_{t,x}[X(T)]-\dfrac{\gamma}{2}\mathrm{Var}^{\mathbb P}_{t,x}[X(T)],~~~ \bm{I}\in \mathcal I,
\end{align}
where $\mathbb E^{\mathbb P}_{t,x}$ and $\mathrm{Var}^{\mathbb P}_{t,x}$ are conditional expectation and conditional variance under condition $X(t)=x$ and $\gamma>0$ is the risk aversion coefficient of the insurer.

From now on, we omit the superscript $\mathbb{P}$ for convenience when an operator is in the sense of probability $\mathbb P$.
 The optimization  problem  of the insurer to be solved is  formulated  as follows. 
\begin{problem}\label{pro:1}The target of the insurer is to find {an} optimal admissible strategy {$\bm{I}^*=\{I^*(t,\cdot)\}_{t\in[0,T]}$ for initial point $(t,x)$} such that
$$J(t,x;\bm{I}^*)=\max_{\bm{I}\in \mathcal  I}J(t,x;\bm{I}).$$
\end{problem}
Particularly, when $\P = \mathbb Q$,
Problem \ref{pro:1} is reduced to an optimal reinsurance problem with
homogeneous beliefs.

\section{Optimal reinsurance  with mean-variance preference}\label{sec:3}
In this section, we first define the equilibrium strategy and demonstrate its existence and uniqueness. We then derive the explicit form of the equilibrium strategy. Finally, we analyze the solution in two special cases of the model.
\subsection{Equilibrium strategy and its existence}
As mentioned in the Introduction, the dynamic mean-variance criterion has the well-known issue of time inconsistency in that
{an} optimal strategy today may not be optimal tomorrow. In this paper, we follow the
mainstream time-consistent approach by treating the decision-making process as a non-cooperative
game against all strategies implemented by future players; see a series of seminal papers on this
topic {by, e.g.,}  \cite{BM14} and \cite{BKM17}, and more recently {by, e.g.,} \cite{DJKX21}, \cite{YHLY23}, and \cite{YWY24}.
\begin{definition}
	 For any given initial point $(t,x)\in[0,T]\times \R$, { $h>0$,} a fixed strategy $\bm{I}$ and an admissible strategy  $\bm{I}^*\in\mathcal I$, we define a perturbed strategy $\bm{I}^{h}$   as
	\begin{align*}
		{I}^{h}(s,y)=\left\{
		\begin{aligned}
			&I(s,y),&t\leqslant s< t+h,\\
			&I^*(s,y),&t+h\leqslant s\leqslant T,
		\end{aligned}
		\right.
	\end{align*}
	{for all $y\geqslant 0$}. If 
	\begin{align*}
	\liminf\limits_{h\rightarrow 0^+}\dfrac{J(t,x;{\bm{I}^*)-J(t,x;\bm{I}}^{h})}{h}\geqslant 0
	\end{align*}
	holds for any admissible strategy $\bm{I} \in \mathcal I$, then  $\bm{I}^*$ is an equilibrium strategy.   The resulting  equilibrium value function is given  by
	\begin{align*}
		V(t,x) = J(t,x;\bm{I}^*).
	\end{align*}
\end{definition}

{For any $I\in\mathcal C$ and $\varphi\in C^{1,1}([0,T]\times \mathbb R)$,  the variational operator $\mathcal{L}^I: C^{1,1}([0,T]\times\mathbb{R}) \to C([0,T]\times\mathbb{R})$ is defined through its image $\mathcal L^I \varphi$:
\begin{align*}
	\mathcal L^I\varphi(t,x) = \dfrac{\partial\varphi}{\partial t}(t,x)+(c-c(I)+rx)\dfrac{\partial\varphi}{\partial x}(t,x)+\mathbb E[\varphi(t,x-(Y-I(Y)))-\varphi(t,x)].
\end{align*}}
We next present the extended HJB system for the characterization of the value function $V$ and the corresponding equilibrium strategy in Theorem \ref{verification}. The proof of this theorem is standard, and thus we simply omit it. The readers are referred to  Theorem 3.1 of \cite{LLX16} and Theorem 5.2 of \cite{BKM17}   for more details.

\begin{theorem}[Verification theorem]\label{verification}
	Suppose that there are two functions $V(t,x)$, $g(t,x)\in C^{1,1}([0,T]\times \mathbb R)$ satisfying the following conditions:
    \begin{itemize}
    \item[{(i)}] For any $(t,x)\in[0,T]\times\R$, 
	\begin{align}\label{hjb1}
		\sup\limits_{I\in\mathcal C} \{\mathcal L^IV(t,x)-\dfrac{\gamma}{2}\mathcal L^Ig^2(t,x)+\gamma g(t,x)\mathcal L^Ig(t,x)\}=0.
	\end{align}
    
	\item[(ii)] For any $(t,x)\in[0,T]\times\R$, 
	\begin{align}\label{hjb2}
	\mathcal L^{I^*}g(t,x)=0,
	\end{align}
and
	\begin{align}\label{hjb0}
		V(T,x) = g(T,x)=x.
	\end{align}
	{where $I^*(t,y)$ is the optimal $I$ to achieve the supremum in \eqref{hjb1} and we use $I^*$ to simplify the presentation whenever there is no risk of confusion.}
    \item[(iii)] %For any $(t, x) \in[0, T] \times \mathbb{R}$, $ I^*$, $\mathcal{L}^{I^*}V(t,x)$ and $\mathcal{L}^{ I^*}g^2(t, x)$ are  deterministic functions of $t$ and independent of $x$.
	{For any $(t, x) \in[0, T] \times \mathbb{R}$, $ I^*$ is deterministic function of $t$ and independent of $x$, that is, $I^*:t\rightarrow I^*(t,\cdot)\in\mathcal C$, and $\bm{I}^*=\{I^*(t,\cdot)\}_{t\in[0,T]}$ is an admissible strategy.}
\end{itemize}
Then {$\bm{I}^*=\{I^*(t,\cdot)\}_{t\in[0,T]}$} is the equilibrium strategy  and  $V(t,x)=J(t,x;\bm{I}^*)$ is the equilibrium value function to  Problem \ref{pro:1}. Besides, $g(t,x)=\mathbb E_{t,x}[X^{*}(T)]$ where $X^*$ is the surplus process under $\bm{I}^*$. 
\end{theorem}

We will first show the
existence and uniqueness of the solution to the system of equations \eqref{hjb1}--\eqref{hjb0}. By Itô's formula and \eqref{dynamic}, we have
\begin{align*}
	\d(e^{r(T-t)}X(t)) = e^{r(T-t)}(c-c(I(t,\cdot)))\d t-e^{r(T-t)}\d\left(\sum\limits_{i=1}^{N(t)}(Y_i-I(\tau_i,Y_i))\right),
\end{align*}
and thus
\begin{align}\label{XT}
	{X(T)}= e^{r(T-t)}{X(t)}+\int_t^Te^{r(T-s)}(c-c(I(s,\cdot)))\d s-\int_t^Te^{r(T-s)}\d\left(\sum\limits_{i=1}^{N(s)}(Y_i-I(\tau_i,Y_i))\right).
\end{align}
Inspired by \eqref{XT}, we conjecture that
the  solution of the extended HJB system of equations has the following form:
\begin{align}
\label{v}V(t,x) = e^{r(T-t)}x+M(t),~~~ g(t,x) = e^{r(T-t)}x+m(t),	
\end{align}with terminal condition $M(T)=m(T)=0$. 
Substituting \eqref{premium} and \eqref{v} 
 into  \eqref{hjb1},
we get 
\begin{align}\label{hjb3}
	M'(t)+ce^{r(T-t)}-e^{r(T-t)}\inf\limits_{I\in\mathcal C}\{(1+\theta)\mathbb E^{\mathbb Q}[I(Y)]+\mathbb E[Y-I(Y)]+\dfrac{\gamma e^{r(T-t)}}{2}\mathbb E[(Y-I(Y))^2] \}=0.
\end{align}
Denote by \begin{equation}\label{eq:H}H(t,I)= (1+\theta)\mathbb E^{\mathbb Q}[I(Y)]+\mathbb E[Y-I(Y)]+\dfrac{\gamma e^{r(T-t)}}{2}\mathbb E[(Y-I(Y))^2].\end{equation}
%and \begin{align}\label{hjb4} d'(t)+ce^{r(T-t)}-e^{r(T-t)}(1+\theta)\mathbb E^{\mathbb Q}[I^*(Y)]-e^{r(T-t)}\mathbb E[Y-I^*(Y)]=0, \end{align} with terminal condition $D(T)=d(T)=0$. $I^*$ in \eqref{hjb4} also realizes the infimum in \eqref{hjb3}.

\begin{proposition}\label{sol_exist}
	{For any $t \in [0,T]$}, there exists an {$I^*(t,\cdot)\in\mathcal C$} that attains the infimum of  \eqref{eq:H}, and the solution is unique in the sense that $\mathbb P(I_1(Y)=I_2(Y))=1$ if both $I_1$ and $I_2$   attains this infimum.
\end{proposition}

Also, by substituting   \eqref{v}  
back into   \eqref{hjb2}, 
we obtain
\begin{align}\label{hjb4} m'(t)+ce^{r(T-t)}-e^{r(T-t)}(1+\theta)\mathbb E^{\mathbb Q}[I^*(t,Y)]-e^{r(T-t)}\mathbb E[Y-I^*(t,Y)]=0. \end{align}    Given the existence of {$I^*(t,\cdot)$}, it is straightforward to verify that both  \eqref{hjb3} and \eqref{hjb4} admit  solutions {as they are ordinary differential equations that can be explicitly solved by direct integration upon substituting $I^*$}. Therefore,   the existence of an equilibrium strategy follows from {Proposition \ref{sol_exist}} and Theorem \ref{verification}.

\begin{remark}
For the existence of the {reinsurance strategy $I^*(t,\cdot)$},  in fact, $\mathcal C$ is a compact subset of $C([0,\infty))$ with respect to the   topology of compact convergence, and $H(t,I)$ is continuous on $\mathcal C$. Thus, by the extreme value theorem, $H(t,I)$ attains both its infimum and supremum on $\mathcal C$.
\end{remark}
\subsection{The form of the equilibrium strategy}\label{sec:3.1}
%The previous discussion only shows that the equilibrium strategy must exist. 
In this section, we investigate the form of the equilibrium strategy as defined in Theorem \ref{verification}.  This leads to the study of the infimum in \eqref{hjb3}, which is expressed as
\begin{align}\label{inf}
	\inf\limits_{I\in\mathcal C}\{H(t, I)\}.
\end{align}
{And \eqref{inf} can be further decomposed as
\begin{align*}
	\inf\limits_{I\in\mathcal C}\{H(t, I)\}=\inf\limits_{a\in[0,\mathbb E(Y)]}\inf\limits_{\substack{I\in\mathcal C,\\ \mathbb E[I(Y)]=a}}H(t,I).
\end{align*}
Since the outer minimization is a single-variable optimization problem, we focus our analysis on the inner problem. Our main result in Theorem 2 below characterizes the form  of its minimizer. Importantly, this characterization is not affected by the subsequent outer minimization, and thus this form can also be viewed as the form of the equilibrium strategy.
The inner minimization is indeed a constrained optimization problem:}
\begin{align*}
	\inf\limits_{I\in\mathcal C}\{H(t,I)\}, ~~
	\text{subject to } \mathbb E[I(Y)]=a, 
\end{align*}
where  $a \in[0, \E[Y]]$.  
With some simplification, the problem becomes
\begin{align}\label{cop}
\begin{split}
	\inf\limits_{I\in\mathcal C}\{(1+\theta)\mathbb E^{\mathbb Q}[I(Y)]+\dfrac{\gamma e^{r(T-t)}}{2}\mathbb E[I(Y)^2]-\gamma e^{r(T-t)}\mathbb E[YI(Y)] \},\\
	\text{subject to } \mathbb E[I(Y)]=a.
\end{split}	
\end{align}
By the Lagrange dual method, problem \eqref{cop} is equivalent to
\begin{align}\label{ldp}
	\inf\limits_{I\in\mathcal C}\{(1+\theta)\mathbb E^{\mathbb Q}[I(Y)]+\dfrac{\gamma e^{r(T-t)}}{2}\mathbb E[I(Y)^2]-\gamma e^{r(T-t)}\mathbb E[YI(Y)]-\lambda\mathbb E[I(Y)] \},
\end{align}
where $\lambda\in\R$ is the Lagrangian multiplier. For convenience, define the functional
\begin{align*}
	{G(t, I;\lambda)}=(1+\theta)\mathbb E^{\mathbb Q}[I(Y)]+\dfrac{\gamma e^{r(T-t)}}{2}\mathbb E[I(Y)^2]-\gamma e^{r(T-t)}\mathbb E[YI(Y)]-\lambda\mathbb E[I(Y)],
\end{align*}
and let $I^*$ be the minimizer of problem \eqref{ldp}. Since $\mathcal C$ is a convex set,  for any $I\in\mathcal C$ and $\varepsilon\in[0,1]$,  we have $I^*+\varepsilon(I-I^*)\in\mathcal C$. Consequently, {$G(t, I^*+\varepsilon(I-I^*);\lambda)$} arrives its minimum value at $\varepsilon=0$. Expanding {$G(t, I^*+\varepsilon(I-I^*);\lambda)$}, we get
\begin{align*}
	{G(t, I^*+\varepsilon(I-I^*);\lambda)}=&\dfrac{\varepsilon^2\gamma e^{r(T-t)}}{2}\mathbb E[((I-I^*)^2]+\varepsilon\gamma e^{r(T-t)}(\mathbb E[I^*(I-I^*)]-\mathbb E[Y(I-I^*)])\\
	&+\varepsilon(1+\theta)\mathbb E^{\mathbb Q}[I-I^*]-\varepsilon\lambda\mathbb E[I-I^*]+constant.
\end{align*}
Thus, {$G(t, I^*+\varepsilon(I-I^*);\lambda)$} is a convex function with respect to $\varepsilon$. The fact that the minimum occurs at $\varepsilon=0$ implies that
\begin{align*}
	\left.\dfrac{\partial {G(t, I^*+\varepsilon(I-I^*);\lambda)}}{\partial\varepsilon}\right|_{\varepsilon=0}\geqslant0.
\end{align*}
This leads to the following condition
\begin{align}\label{dG}
\begin{split}
	&\gamma e^{r(T-t)}\mathbb E[I^*\cdot I^*]	-\gamma e^{r(T-t)}\mathbb E[YI^*]+(1+\theta)\mathbb E^{\mathbb Q}[I^*]-\lambda\mathbb E[I^*]\\
	\leqslant &\gamma e^{r(T-t)}\mathbb E[I^*\cdot I]	-\gamma e^{r(T-t)}\mathbb E[YI]+(1+\theta)\mathbb E^{\mathbb Q}[I]-\lambda\mathbb E[I].
\end{split}
\end{align}
Obviously, \eqref{dG} is equivalent to
\begin{align}\label{argmin}
	I^*\in \argmin_{I\in\mathcal C}\{\gamma e^{r(T-t)}\mathbb E[I^*\cdot I]	-\gamma e^{r(T-t)}\mathbb E[YI]+(1+\theta)\mathbb E^{\mathbb Q}[I]-\lambda\mathbb E[I]\}.
\end{align}
Finally, since any function $I\in\mathcal C$ is absolutely continuous and  almost everywhere differentiable,  we can establish the lemma below by differentiating with respect to $I$ to  characterize $I^*$.
\begin{lemma}
    \label{L}
	Define $L(s; I^*,\lambda)$ as
	\begin{align*}
		L(s; I^*,\lambda)=\int_s^{\infty}(\gamma e^{r(T-t)}I^*(t,y)-\gamma e^{r(T-t)}y-\lambda)\d F(y)+(1+\theta)S^{\mathbb Q}(s).
	\end{align*}
	Then $I^*$ is the solution of \eqref{ldp} if and only if
	\begin{align}\label{dI}
		\frac{\partial I^*}{\partial y}(t,s)=\chi_{D^-}(s)+\xi(s){\chi_{D^0}(s)},
	\end{align}
	where $\chi$ is the indicator function, and 
	\begin{align*}
		D^-=\{s:L(s; I^*,\lambda)<0\},\ D^0=\{s:L(s;I^*,\lambda)=0\},
	\end{align*}
	and $\xi(s)$ is a function such that $I^*\in\mathcal C$.
\end{lemma}

	Although Lemma \ref{L} provides a necessary and sufficient condition for $I^*$ to be the solution of equation \eqref{ldp} and \eqref{argmin}, it cannot be directly used to determine 
 $I^*$, since both sides of equation \eqref{dI} involve $I^*$. Nevertheless, it serves as a useful tool for analyzing the form of
 $I^*$.  In the following analysis, we will derive the form of $I^*$  by examining the sign changes of 
 $L(s;I^*,\lambda)$. 
 
To proceed, we define the likelihood ratio. A Borel measurable function $\mathrm {LR}$ is called a likelihood ratio  of $\mathbb Q$ against $\mathbb P$ if it satisfies  for all  Borel measurable set $E$, 
$$
\mathbb  Q(Y \in E \cap \{\mathrm{LR}<\infty\} )=\int_E  \mathrm{LR}(y)  \d F(y).
$$
It is clear that $\mathrm {LR}(y) \geqslant 0$ almost surely.  
\begin{assumption}\label{ass1}Suppose that  the  likelihood ratio 
LR has finite variation.
\end{assumption}
If $\mathbb Q \ll \mathbb P$,\footnote{Note that $\mathbb Q \ll \mathbb P$ means that the measure $\mathbb{Q}$ is absolutely continuous with respect to the measure $\mathbb{P}$. That is, for any measurable set $A$, if $\mathbb{P}(A) = 0$, then $\mathbb{Q}(A) = 0$.} which is a special case, $\mathrm{LR}$ reduces to the Radon-Nikodym derivative $\d F^{\mathbb Q} / \d F$.  Suppose  further  that the claim {$Y$} has a density under both probabilities,  $\mathbb P$ and $\mathbb Q$   with corresponding probability density functions  $f$ and $f^{\mathbb Q}$. 
Given $f(y)>0$ for $y>0$, we have \begin{equation}\label{eq:RN}
	\text{LR}(y)=\frac{f^{\mathbb Q}(y)}{f (y)}.\end{equation} 
  Additionally, Assumption \ref{ass1} also  allows for the possibility that $Y$  have a probability mass at some point. For example, let  the cumulative distribution function of $Y$  be written as $ \mathbb{P}\{Y \leqslant t\}=p_0+(1-p_0)\int_0^t f (s) d s$, and $\mathbb{Q}\{Y \leqslant t\}=q_0+(1-q_0)\int_0^t f^{\mathbb Q}(s) d s, $ where $p_0, q_0 \in[0,1)$. In this case, 
$$\LR(y)=\frac{q_0 \delta(0)+\left(1-q_0\right) f^{\mathbb Q}(y)}{p_0 \delta(0)+\left(1-p_0\right) f (y)}= \begin{cases}\frac{q_0}{p_0}, & y=0, \\ \frac{1-q_0}{1-p_0} \cdot \frac{f^{\mathbb Q}(y)}{f(y)}, & y>0,\end{cases}$$
 where $\delta$ refers to the Dirac delta measure.
%\begin{remark} Studies on belief heterogeneity typically assume the existence and continuity of the likelihood ratio function, often even imposing monotonicity. Specifically, the decreasing LR indicates that the reinsurer is more optimistic about the size of the underlying loss compared to the insurer, while the increasing LR  suggests that the insurer holds a more optimistic view than the reinsurer. However, such a framework does not capture the full spectrum of belief heterogeneity between the insured and the insurer. In contrast, we adopt the definition of the likelihood ratio function from \cite{EG79} without imposing monotonicity, which allows for greater flexibility. In particular, we permit the insurable loss to have an atom at zero, a condition that closely mirrors real-world insurance practices, as emphasized by \cite{S68}. This approach also accommodates two widely used risk measures, VaR and ES, as potential premium models, as discussed in Section \ref{sec:4}. \end{remark}

First, we relax the incentive compatibility condition and aim to minimize  \eqref{ldp} over a larger admissible set.
Let 
\begin{align}\label{eq:Beta}
	\mathcal B :=\{f:[0,\infty)\rightarrow[0,\infty)|0\leqslant f(y)\leqslant y,  \forall y\geqslant 0 \}.
\end{align}
{Clearly, we have $\mathcal C\subset \mathcal B$.   In this case,   \eqref{ldp} becomes}
\begin{align}\label{ldpb}
\begin{split}
	&\inf\limits_{I\in\mathcal B}\{(1+\theta)\mathbb E^{\mathbb Q}[I(Y)]+\dfrac{\gamma e^{r(T-t)}}{2}\mathbb E[I(Y)^2]-\gamma e^{r(T-t)}\mathbb E[YI(Y)]-\lambda\mathbb E[I(Y)] \} \\
	=&\inf\limits_{I\in\mathcal B}\left\{\int_0^{\infty}\left(\dfrac{\gamma e^{r(T-t)}}{2}I(y)^2-\gamma e^{r(T-t)}yI(y)-\lambda I(y)+(1+\theta)I(y)\text{LR}(y)\right)\d F(y)\right\}.
\end{split}
\end{align}
This problem admits a pointwise optimal solution. Let  $$h(t, z)=\dfrac{\gamma e^{r(T-t)}}{2}z^2-\gamma e^{r(T-t)}yz-\lambda z+(1+\theta)\text{LR}(y)z,$$ and  we can solve problem \eqref{ldpb} pointwise by minimizing $h(t, z)$ with respect to 
$z$ over  $[0,y]$. Since $h(t, z)$ is a convex function in $z$,   the first-order condition gives 
{\begin{align*}
	\frac{\partial h(t, z)}{\partial z}= \gamma e^{r(T-t)}z-\gamma e^{r(T-t)}y-\lambda+(1+\theta)\text{LR}(y)=0,
\end{align*}}
which simplifies to
\begin{align}\label{phiy}
	 z=y+\dfrac{\lambda-(1+\theta)\text{LR}(y)}{\gamma e^{r(T-t)}}=:\phi_{\lambda}({t},y).
\end{align}
{Thus, the solution to \eqref{ldpb} is given by  \begin{equation}\label{eq:tildeI}\widetilde{I}(t,y)=\min\{y,\max\{0,\phi_{\lambda}({t},y)\}\}.\end{equation}
 Since  $\text{LR}(y)$  has finite variation, it implies that  $\widetilde{I}(t,\cdot)$ also has finite variation. In general, $\widetilde{I}(t,\cdot)$ does not {belong} to $\mathcal C$ unless  $\widetilde{I}(t,\cdot)$ is continuous and $0\leqslant\widetilde{I}'(t,\cdot)\leqslant 1$ holds for all differentiable $y\geqslant 0$, or  equivalently,  $\phi_{\lambda}'(t,\cdot)\in [0,1]$ for all differentiable $y\geqslant 0$. In this case,  $\widetilde{I}(t,y)$ is the solution to \eqref{ldp}.}

{We adopt an optimization approach over a partitioned domain as mentioned in the Introduction.
To proceed, we make the following two assumptions.
\begin{assumption}\label{ass2}
	Suppose that  the  likelihood ratio $\mathrm{LR}(y)$ and $F(y)$ are non-differentiable at only finitely many points.
\end{assumption}
Denote the non-differentiable points in Assumption \ref{ass2} by $y'_1,y'_2,...,y'_n$, where $0=y'_0\leqslant y'_1<y'_2<\cdots<y'_n<\infty$.  Then, the function $\phi_{\lambda}(t,y)$ is differentiable with respect to $y$ on each interval $(y'_{i-1},y'_{i})$ for $i=1,...,n$, as well as on  $(y'_n,\infty)$. Building on this, we impose the following additional assumption.
\begin{assumption}\label{ass3}
	For any $t\in[0,T]$, we assume that each interval $(y'_{i-1},y'_{i})$ for $i=1,...,n$ and $(y'_n,\infty)$ can be partitioned into finite disjoint sub-intervals according to the values of the derivative $\frac{\partial \phi_{\lambda}(t,y)}{\partial y}$.  That is, on each sub-interval, one of the following holds: $$\frac{\partial \phi_{\lambda}(t,y)}{\partial y}>1, \frac{\partial \phi_{\lambda}(t,y)}{\partial y}\in[0,1], ~\text{or}~ \frac{\partial \phi_{\lambda}(t,y)}{\partial y}<0.$$
\end{assumption}

 {In the remainder of this paper, unless otherwise stated, we always assume that Assumptions \ref{ass1}--\ref{ass3} holds.}

 Let $0 = y_0(t) < y_1(t) < \cdots < y_{m_t}(t)$ be the strictly increasing sequence comprising the point 0 together with all sub-interval endpoints arising from Assumptions \ref{ass2} and \ref{ass3}.  For convenience, we simplify $y_i(t)$ to $y_i$ without causing confusion and set $y_{m_t+1} = \infty$. This induces a partition of the interval $[0,\infty)$ into disjoint sub-intervals $S_{i,j_i}$, where $i=1,2,...,m_t+1$ and $j_i\in\{1,2,3\}$, defined according to the values of $\frac{\partial \phi_{\lambda}(t,y)}{\partial y}$ as follows:

\begin{align*}
	[0,\infty)=\bigcup\limits_{i=1}^{m_t+1}S_{i,j_i},
\end{align*}
where $S_{i,j_i}$ is $[y_{i-1},y_i)$ and
\begin{align*}
	j_i=\left\{
	\begin{aligned}
		&1,\text{ if }\tfrac{\partial \phi_{\lambda}(t,y)}{\partial y}>1\text{ on }(y_{i-1},y_i),\\
		&2,\text{ if }\tfrac{\partial \phi_{\lambda}(t,y)}{\partial y}\in[0,1]\text{ on }(y_{i-1},y_i),\\
		&3,\text{ if }\tfrac{\partial \phi_{\lambda}(t,y)}{\partial y}<0\text{ on }(y_{i-1},y_i).
	\end{aligned}
	\right.
\end{align*}}
With this partition in place, we can now apply Lemma \ref{L} to derive the main theorem below.
\begin{theorem}\label{form}
	{Under Assumptions \ref{ass1}--\ref{ass3},} when $F(y)$ is strictly increasing, for problem \eqref{ldp}, the  optimal indemnity function  $I^*$  over $S_{m_t+1,j_{m_t+1}}$ at $t$ is given by 
	\begin{enumerate}[label*=(\roman*)]
		\item If $j_{m_t+1}=1$, then $I^*(t,y)=I^*(t,y_{m_t})+(y-s_{m_t+1})_+$ for some $s_{m_t+1}\in[y_{m_t},\infty)$.
		\item If $j_{m_t+1}=2$, then $I^*(t,y)=\min\{I^*(t,y_{m_t})+y-y_{m_t},\max\{I^*(t,y_{m_t}),\phi_{\lambda}(t,y) \} \}$.
		\item If $j_{m_t+1}=3$, then $I^*(t,y)=I^*(t,y_{m_t})+Z_{(y_{m_t},s_{m_t+1}]}(y)$ for some $s_{m_t+1}\in[y_{m_t},\infty)$.
	\end{enumerate}
	The optimal indemnity function $I^*$ over $S_{i,j_{i}}$ at $t$, where  $i=1,2,...,m_t$, is given by 
	\begin{enumerate}[label*=(\roman*), resume]
		\item If $j_{i}=1$, then $I^*(t,y)=I^*(t,y_{i-1})+Z_{(s_{i},s_{i}+I^*(t,y_i)-I^*(t,y_{i-1})]}(y)$ for some $s_{i}\in[y_{i-1},y_i]$.
		\item If $j_{i}=2$, then $I^*(t,y)=\min\{\max\{\phi_{\lambda}(t,y),I^*(t,y_i)+y-y_i,I^*(t,y_{i-1})\},I^*(t,y_{i-1})+y-y_{i-1},I^*(t,y_i) \}$.
		\item If $j_{i}=3$, then $I^*(t,y)=I^*(t,y_{i-1})+y-y_{i-1}-Z_{(s_{i},s_{i}+y_i-y_{i-1}-I^*(t,y_i)+I^*(t,y_{i-1})]}(y)$ for some $s_{i}\in[y_{i-1},y_i]$.
	\end{enumerate}
	Here, $Z_{(a,b]}(x)=(x-a)_+-(x-b)_+$.
\end{theorem}

{It is important to emphasize that Theorem \ref{form}, in fact, does not provide the values of $I^*(t,\cdot)$ on the entire intervals; rather, it gives a parametric form of $I^*(t,\cdot)$. In particular, the values of $I^*(t,\cdot)$ at the interval endpoints remain unknown and should be treated as parameters. Thus,  Theorem \ref{form} alone is insufficient to fully construct the optimal strategy  $I^*(t,\cdot)$. Instead,  it expresses $I^*(t,\cdot)$  in terms of  parameters $ I^*(t,y_1), I^*(t,y_2), \ldots, I^*(t,y_{m_t})$, the multiplier $\lambda$, and the interval-specific parameters $s_i$. By substituting this parametric form  into \eqref{inf}, the original infinite-dimensional optimization problem of finding the function $I^*(t,\cdot)$ is transformed into a finite-dimensional problem of finding the  optimal values of these parameters. Solving this finite-dimensional problem yields the minimizing parameters and thus fully $I^*(t,\cdot)$.}

	%Theorem \ref{form}  effectively simplifies   the complexity of problem \eqref{ldp} by providing a specific form of the solution. Specifically,  we have reduced the infinite-dimensional problem of finding a function $I(y)$ to the   finite-dimensional problem of finding the  optimal values of the parameters as described in the six cases. 
    
    We provide an example of ${I^*(t,\cdot)}$ in Figure \ref{fig:example}. In this case, the interval $[0,\infty)$ is divided into three sub-intervals, and  five parameters are required to determine ${I^*(t,\cdot)}$: $s_1$, $s_3$, ${I^*}(t,y_1)$, ${I^*}(t,y_2)$ and $\lambda$.  In general, when there are $n$ sub-intervals,  at most $2n+1$ parameters are needed to fully determine ${I^*(t,\cdot)}$. Moreover, from \eqref{phiy} and parts $(ii)$ and $(v)$ of Theorem \ref{form}, it is easy to know that there exists $\lambda_{\min}$ and $\lambda_{\max}$ such that ${I^*(t,y)}$ remains unchanged when $\lambda\geqslant \lambda_{\max}$ or $\lambda\leqslant \lambda_{\min}$ (which corresponds to the  two diagonal dashed lines  in the middle region). Consequently, we can  further define the domain of $\lambda$ as   $[\lambda_{\min},\lambda_{\max}]$. 
    
    Furthermore, the division of sub-intervals depends on $\phi_{\lambda}(t,y)$, which actually depends on $\text{LR}(y)$. Specifically, we have the following relationships
	\begin{align*}
	\tfrac{\partial \phi_{\lambda}(t,y)}{\partial y}>1&\iff\text{LR}'(y)<0,\\
		0\leqslant \tfrac{\partial \phi_{\lambda}(t,y)}{\partial y} \leqslant 1&\iff 0\leq \text{LR}'(y)\leq\frac{\gamma}{1+\theta}e^{r(T-t)},\\
		\tfrac{\partial \phi_{\lambda}(t,y)}{\partial y}<0&\iff \text{LR}'(y)>\frac{\gamma}{1+\theta} e^{r(T-t)}.
	\end{align*}

\begin{figure}[!htbp]
\centering
\includegraphics[scale=0.6]{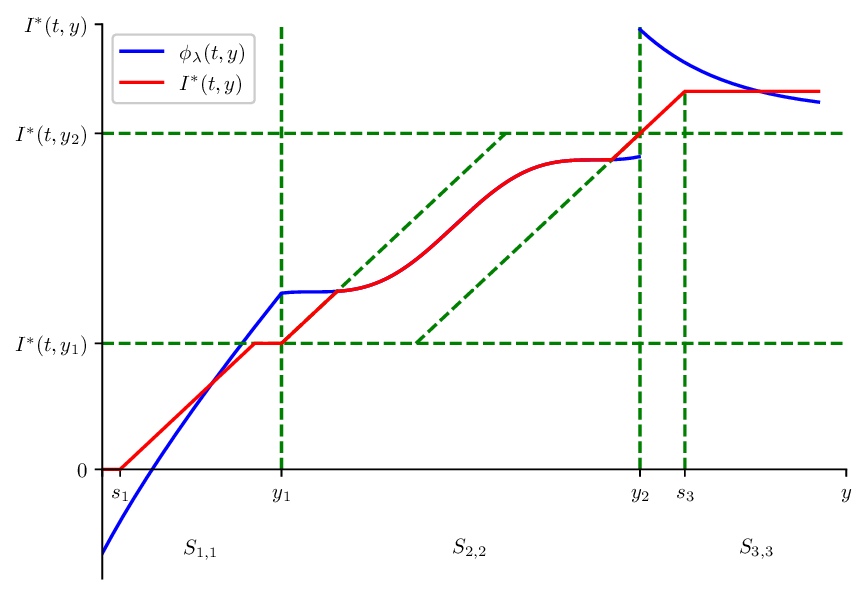}
\caption{An example of   the optimal  indemnity function $I^*$}
\label{fig:example}
\end{figure}

We observe that the condition of $\tfrac{\partial \phi_{\lambda}(t,y)}{\partial y}>1$ is independent  of $t$.  Further,  the relationships outlined above provide critical insights into the risk attitudes of the insurer and reinsurer, which, in turn, influence the optimal structure of the reinsurance contract. Specifically:
\begin{itemize}\item When $\tfrac{\partial \phi_{\lambda}(t,y)}{\partial y} > 1$, i.e, when $\mathrm{LR}(y)$ is decreasing in $y$, it indicates that the insurer is less optimistic about the right tail risk compared to the reinsurer.   This typically results in the insurer opting for a limited excess-of-loss   strategy (locally) for that sub-interval, transferring the risk associated with larger losses to the reinsurer. This is illustrated in Case (iv) of Figure \ref{fig:1}.

\item When $0\leqslant \tfrac{\partial \phi_{\lambda}(t,y)}{\partial y}\leqslant 1$, i.e., $\mathrm{LR}(y)$ is increasing in $y$,   the insurer is more optimistic than the reinsurer about the right tail risk, but not excessively so, as there is an upper bound  of  the derivative determined by the parameters   $\gamma$, $\theta$, $r$ and $t$. In this scenario, the insurer may choose to implement a co-insurance strategy; see  Case (v) in Figure \ref{fig:1}.

\item  When $\tfrac{\partial \phi_{\lambda}(t,y)}{\partial y} < 0$, the likelihood ratio derivative $\mathrm{LR}'(y)$ exceeds a positive certain threshold,   the insurer is significantly more optimistic about the right tail risk than the reinsurer. In this case,  the insurer is willing to transfer the lower tail risk to the reinsurer as the insurer is relatively more pessimistic about smaller losses.  {At the same time, the insurer may opt to purchase limited reinsurance when the claim is not excessively large}; see Case (vi) in Figure \ref{fig:1}.

\end{itemize}
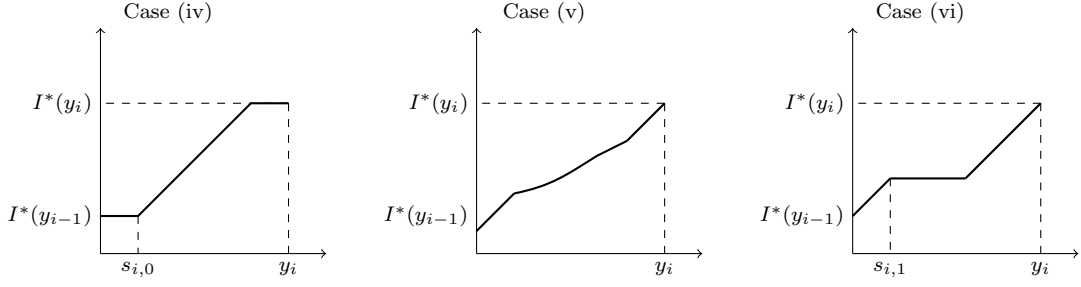
\begin{figure}
    \centering
\begin{tikzpicture}
    \scriptsize
    \begin{scope}[xshift=0cm]
        % Axes
        \draw[->] (0,0) -- (3,0) node[right] {};
        \draw[->] (0,0) -- (0,3) node[above] {};
        \draw[dashed] (2.5,2) -- (0,2) node[left] {$I^*(t,y_i)$};
        \draw (0,0.5) node[left] {$I^*(t,y_{i-1})$};
        \draw[thick] (0,0.5) -- (0.5,0.5) -- (2,2)--(2.5,2);
        \draw[dashed]  (2.5,2)--(2.5,0)node[below] {$y_i$};
        \draw[dashed] (0.5,0.5) -- (0.5,0) node[below] {$s_{i}$};
            \node at (0.9,3.2) {Case (iv)};
    \end{scope}

    \begin{scope}[xshift=5cm]
        % Axes
       \draw[->] (0,0) -- (3,0) node[right] {};
        \draw[->] (0,0) -- (0,3) node[above] {};

       \draw[thick] (0,0.3) -- (0.5,0.8);  
        % Lines
        
\draw[thick] (0.5,0.8) .. controls (1,0.9) and (1.2,1.05) .. (1.6,1.3) .. controls (1.8,1.4) and (2,1.5) .. (2,1.5);

 \draw (0,0.5) node[left]{$I^*(t,y_{i-1})$}; 
 
 \draw[thick] (2,1.5) -- (2.5,2) node[right] {};
       \draw[dashed] (2.5,2) -- (2.5,0) node[below]{$y_i$}; 

       \draw[dashed] (2.5,2) -- (0,2) node[left]{$I^*(t,y_i)$}; 
         \node at (0.9,3.2) {Case (v)};
    \end{scope}

    \begin{scope}[xshift=10cm]
       
       \draw[->] (0,0) -- (3,0) node[right] {};
        \draw[->] (0,0) -- (0,3) node[above] {};
        
       \draw[thick] (0,0.5) -- (0.5,1);
       \draw (0,0.5) node[left] {$I^*(t,y_{i-1})$};
       \draw[thick] (0.5,1) -- (1.5,1);
 \draw[dashed] (0.5,1) -- (0.5,0) node[below]{$s_{i}$}; 
        \draw[thick] (1.5,1) -- (2.5,2);     
       \draw[dashed] (2.5,2) -- (2.5,0) node[below]{$y_i$}; 
       \draw[dashed] (2.5,2) -- (0,2) node[left]{$I^*(t,y_{i})$}; 
        \node at (0.9,3.2) {Case (vi)};
    \end{scope}
\end{tikzpicture}

    \caption{The form of the optimal  indemnity function $I^*(t,\cdot)$}
    \label{fig:1}
\end{figure}
Note that the cases (i) corresponds to (iv), (ii) to (v), and (iii) to (vi), with the distinction that cases (i)-(iii) represent the final partition, and therefore do not involve boundaries.

\begin{remark}\label{rem:2}   
The threshold $\frac{\gamma}{1+\theta} e^{r(T-t)}$ reflects the insurer's level of optimism regarding the underlying risk  in a certain sense.    The smaller this value, the more likely the insurer is to become relatively optimistic. 
 First,  $\gamma$  reflects the insurer's level of risk aversion. A higher value of $\gamma$ indicates a more risk-averse insurer,  making him less willing to accept large risks. This leads to an increase in the threshold, meaning the  insurer becomes less optimistic about the right tail.
 Second, a higher $\theta$ indicates a greater margin  on the premium.  From the insurer's perspective,   a higher $\theta$ also signals a higher cost of reinsurance, which may increase the insurer's optimism as they are less willing to accept higher premiums.
Third,  the risk-free interest rate 
$r$ suggests the returns the insurer can earn on capital. A higher 
$r$ leads to a more cautious stance and raises the threshold, as the insurer prefers to allocate more capital towards investments rather than taking on significant risks. 
 Finally, as the time $t$  approaches $T$,   the remaining time to maturity decreases. This leads to an increase in the insurer's optimism about the right tail risk, as the insurer becomes more confident that larger risks will materialize in the short term, thereby lowering the threshold.
\end{remark}

In the following proposition, we present the case where LR is monotonically decreasing.   By Theorem \ref{form}, the result is straightforward.

\begin{proposition}\label{prop:3} If $\mathrm{LR}(x)$ is decreasing over  $[0, \infty)$, then the optimal indemnity function takes the form {$I^*(t,y)=(y-d_t)_{+}$ for some $d_t > 0$}.
\end{proposition}

Theorem \ref{form} provides the form of equilibrium strategy at time $t$. In general, the structure of the equilibrium strategy may change as $t$ changes, since the partition of $[0,\infty)$ depends on $t$. Nonetheless, under some mild assumptions,  we can establish that the equilibrium strategy  $I^*(t,y)$ is  both unique and continuous in $t$.
\begin{theorem}\label{thm3}
	Under Assumptions \ref{ass1}--\ref{ass3}, if the cumulative distribution function $F(y)$ of $Y$ is {strictly} increasing and $\mathrm{LR}(y)$ is piece-wise $C^1$, then the optimal indemnity function  $I^*(t, y)$ is unique and continuous in $t$.
\end{theorem}

\subsection{Two simplified models}
In this subsection, we examine two special cases of the model: one where the insurer and reinsurer share the identical belief, and another where the incentive compatibility constraint is relaxed. %First, we present the optimal reinsurance strategy under the assumption of homogeneous beliefs.
\begin{corollary}\label{coro:1} If $\mathbb P=\mathbb Q$, the optimal  solution to Problem \ref{pro:1} is given by $$ \overline I^*(t,y)=(x- d^*_t)_+,$$ where $d^*_t=\frac{\theta} {\gamma e^{r(T-t)}}.$
\end{corollary}

{Corollary  \ref{coro:1} can be viewed as a special case of Proposition \ref{prop:3}  when $\mathrm{LR} = 1$, and we can solve the $d^*_t$ explicitly.   It is worth noting that    the result in Corollary \ref{coro:1} aligns with that in \cite{CS19}, which explores a more general problem under the assumption of homogeneous beliefs. In their work, stochastic Stackelberg differential reinsurance games are investigated, where the reinsurer also determines the optimal reinsurance contract. In the case of homogeneous beliefs, the excess-of-loss strategy is automatically satisfied, and the optimal solution remains unchanged, regardless of the assumption of incentive compatibility.}

In the case of  belief heterogeneity,  if the incentive compatibility constraint is removed, then Problem \ref{pro:1} becomes: 
\begin{problem}\label{prob:2} {The target of the insurer is to find the optimal admissible strategy $\widetilde{\bm{I}}^*=\{\widetilde I^*(s,\cdot)\}_{s\in[t,T]}$ for initial point $(t,x)$} such that
$$ J(t,x;\widetilde{\bm I}^*)=\max_{\bm{I}\in \mathcal I_0} J(t,x; \bm{I}),$$
where $\mathcal I_0$ denotes the collection of all admissible strategies obtained by replacing the constraint set $\mathcal C$ in Definition \ref{def:1} with $\mathcal B$.
 \end{problem}
%where $\mathcal B$ is defined by \eqref{eq:Beta}. 
 
By  the analysis above,  to solve Problem \ref{prob:2}, we  need   to solve  \eqref{ldpb} firstly, and the solution has already  been given by \eqref{eq:tildeI}. We summarize it in the following corollary.
\begin{corollary}\label{coro:2}
For problem  \eqref{ldpb}, the  optimal indemnity function  $\widetilde I^*$  is given by   $$\widetilde{I}^*(t,y)= \min\{y,\max\{0,\phi_{\lambda}(t,y)\}\}.$$  
\end{corollary}
Even in the absence of the incentive compatibility constraint, the context of continuous-time mean-variance models with belief heterogeneity has not been thoroughly explored in the existing literature. In contrast to Theorem \ref{form}, we observe that the optimal indemnity function simplifies significantly, taking a single form over the entire interval $[0, \infty)$. We can further discuss the solution to Problem \ref{prob:2} based on the specific LR. Note that,  unlike the case with homogeneous beliefs in Corollary \ref{coro:1}, the strategy $\widetilde{I}^*(t,y)$ may lead to moral hazard;  see more details in Section \ref{sec:5}.

\section{\texorpdfstring{Optimal {reinsurance} for special forms of belief heterogeneity}{Optimal reinsurance for special forms of belief heterogeneity}}\label{sec:4}
In Section \ref{sec:3}, we provide the explicit form of the constrained problem. That is, for Problem  \ref{pro:1},  we show that for any $ I \in \mathcal{C} $, there always exists an $ I^* \in \mathcal{C} $  as described in   Theorem \ref{form}  such that $H(t,I^*) \leqslant H(t,I)$.   In this section, we  further simplify  the strategy's form and identify the optimal value of its parameters for some special cases. For simplicity, we assume that the distribution function $F$ of $Y$ has density $f$ in this section, although this assumption can be modified to include a discrete jump.

A distortion risk measure $\rho$ for a random variable $X$ is  defined as
\begin{align*}
	\rho(X)=\int_0^{\infty}g(S(y))\d y-\int_{-\infty}^0(1-g(S(y)))\d y,
\end{align*}
where $g$ is a non-decreasing function on $[0,1]$ called the distortion function,   satisfying $g(0)=0$ and $g(1)=1$. Since the claim $Y$ is a non-negative   in our context,  we have 
\begin{align*}
	\rho(Y)=\int_0^{\infty}g(S(y))\d y.
\end{align*}
Let $F^{\mathbb Q}(y)=1-g(S(y))$. It then  follows that \begin{align*}
	(1+\theta)\mathbb E^{\mathbb Q}[Y]=(1+\theta)\int_0^{\infty}S^{\mathbb Q}(y)\d y=(1+\theta)\int_0^{\infty}g(S(y))\d y=(1+\theta)\rho(Y),
\end{align*}
which is exactly Wang's   premium principle (\cite{WYP97}), and can be seen as a  manifestation of belief heterogeneity.
Since $ g $ is the distortion function applied by the reinsurer, if $ g $ is convex over $[0,1]$, the reinsurer is more optimistic about the potential loss. Conversely, if $ g $ is concave, the reinsurer is less optimistic about the potential loss.

The following proposition provides the optimal reinsurance  strategy when 
$g$ is convex.
\begin{proposition}\label{cov-dis}
Assume that the distortion $g$ is convex and differentiable.  Then   the optimal indemnity function for Problem \ref{pro:1} is given by   $I^*(t, y)=(y-d^*_t)_+,$
where
$$d^*_t=\inf\left\{d\geq0:~1+\gamma e^{r(T-t)}d-(1+\theta)\dfrac{g(S(d))}{S(d)}\geq0\right\}.$$
\end{proposition} 

{\begin{remark} Since $g$ is convex, we have $g(S(d))/S(d)\leq 1$. Compared to the case of  homogeneous beliefs in Corollary \ref{coro:1}, it is clear that $d^*_t\leqslant \overline d^*_t$.  This implies that the insurer always chooses to transfer more risk to the reinsurer under belief heterogeneity when $g$ is convex. This observation is expected, as a convex  $g$ suggests that the reinsurer is more optimistic about the potential loss and thus charges a relatively cheaper premium. 

Additionally, we observe that $d^*_t$ is increasing in $t$.  This is because, as $t$ approaches $T, $ the game is closer to its end, and the insurer becomes more optimistic about the  risk and  chooses to retain more risk, spending less on reinsurance. This behavior aligns with the case of homogeneous beliefs.
\end{remark}}

As we {know}, Value at Risk (VaR) and Expected Shortfall (ES) are two special distortion risk measures,  where the VaR  at level $\alpha\in(0,1)$  is defined as
\begin{align*}
	\VaR_{\alpha}(Y)=\inf\{y\in\R_+: F(y)\geq 1-\alpha\}, 
\end{align*}  and  the ES at  level $\alpha\in(0,1)$ is the functional $\ES_\alpha:L^1 \to \mathbb{R}$ defined by
\begin{align*}
	\mathrm{ES}_{\alpha}(Y)=\dfrac{1}{\alpha}\int_0^{\alpha}\VaR_s(Y)\d s,
\end{align*}
 and  $\ES_{0}(Z)=\text { ess-sup }(Z)=\VaR_{0}(Z)$  which may be infinite.   In particular, we have $g(y)=\chi_{(\alpha,1]}(y)$ for $\VaR_\alpha$ and $g(y)=\dfrac{y}{\alpha}\chi_{[0,\alpha)}(y)\wedge 1$ for $\ES_\alpha$.  In the following,  the explicit solutions are  derived when the  the reinsurer uses VaR  and ES  as the premium principles.

We first focus on the case of VaR.  
Since $g(y)=\chi_{(\alpha,1]}(y)$, we have  $F^{\mathbb Q}(y)=\chi_{\{F(y)\geqslant 1-\alpha\}}(y)$ and $f^{\mathbb Q}(y)=\delta(y-\VaR_{\alpha}(Y))$, where $\delta$ is the Dirac measure. For measure $\mathbb Q$ and Borel measurable set $A$, we have
\begin{align*}
	0=\mathbb Q(Y\in A\cap \{y\neq \VaR_{\alpha}(Y)\})=\int_{A}0\d F(y).
\end{align*}
Thus, we can define LR$(y)$ as follows
\begin{align*}
	\text{LR}(y)=\left\{
	\begin{aligned}
		&0,&\text{ if } y\neq \VaR_{\alpha}(Y),\\
		&\infty,&\text{ if } y= \VaR_{\alpha}(Y).
	\end{aligned}\right.
\end{align*}
\begin{proposition}\label{prop:4}
Assume that $g(y)=\chi_{(\alpha,1]}(y)$, i.e, the premium is calculated by $\VaR$. Then the optimal indemnity function for Problem \ref{pro:1} is given by     \begin{equation}\label{eq:dual1} I^*(t,y)=y\wedge a^*_t+(y-\VaR_\alpha(Y))_+,\end{equation}where $$a^*_t=\inf\left\{a\in[0, \VaR_\alpha(Y)]: ~ \theta+F(a)+\gamma e^{r(T-t)}\left(aS(a)-\alpha\VaR_{\alpha}(Y)-\int_a^{\VaR_{\alpha}(Y)}ydF(y)\right)\geq0\right\}. $$ 
\end{proposition}

\begin{remark} 
If the reinsurance premium is determined using VaR, the optimal indemnity function takes the form of a dual truncated excess-of-loss coverage, as given in \eqref{eq:dual1}. Specifically, the insurer transfers all risk to the reinsurer up to a certain threshold $a^*_t$, and partially transfers losses exceeding the deductible $\operatorname{VaR}_\alpha(Y)$. This means that the retained loss after reinsurance is bounded by the condition $Y - I^*(t,Y) \leqslant \VaR_\alpha(Y) - a^*_t$. This structure is reasonable because the premium is based on VaR, and any losses exceeding the VaR are essentially covered at no additional cost. Moreover, from the proof of Proposition \ref{prop:4}, we observe that if $\theta \geqslant \gamma e^{r(T-t)} \left( \mathbb{E}[Y] + \alpha\operatorname{VaR}_\alpha(Y) - \alpha\operatorname{ES}_\alpha(Y) \right)$, then $ a^*_t = 0 $, meaning that the optimal indemnity function takes the form of  excess-of-loss. This occurs because, when the premium is relatively higher, the insurer will choose not to buy reinsurance for smaller losses.
\end{remark}

Next, we continue with  ES, i.e.,  the   distortion function is given by 
\begin{align*}
	g(y)=\dfrac{y}{\alpha}\chi_{[0,\alpha)}(y)+\chi_{[\alpha,1]}(y).
\end{align*}
In this case,  we have $$F^{\mathbb Q}(y)=1-g(S(y))=\chi_{[\VaR_{\alpha}(Y),\infty)}(y)(F(y)+\alpha-1)/\alpha.$$  Further, the  likelihood ratio  function is   defined as
\begin{align*}
	\mathrm{LR}(y)=\left\{
	\begin{aligned}
		&0,&y &\in[0,\VaR_{\alpha}(Y)),\\
		&\dfrac{1}{\alpha},&y &\in [\VaR_{\alpha}(Y),\infty).
	\end{aligned}
	\right.
\end{align*}
\begin{proposition}\label{prop:5}
 Assume that $g(y)=\chi_{(\alpha,1]}(y)$, i.e.,  the premium is calculated by $\ES$. Then  the optimal form of the  indemnity function for Problem \ref{pro:1}  is given by    
\begin{align*}
	I^*(t, y)=y\wedge a^*_t+(y-b^*_t)_+.
\end{align*} Moreover, let $I^*_{a,b}(y)=y\wedge a+(y-b)_+$, then the optimal parameters $(a^*_t,b^*_t)$ are given by
$$(a^*_t,b^*_t)\in \argmin\limits_{a,b}\left\{(1+\theta)\ES_\alpha(I^*_{a,b}(Y))+\mathbb E[Y-I^*_{a,b}(Y)]+\dfrac{\gamma e^{r(T-t)}}{2}\mathbb E[(Y-I^*_{a,b}(Y))^2]\right\},$$
subject to $0\leqslant a\leqslant \VaR_{\alpha}(Y)$ and $\VaR_{\alpha}(Y)\leqslant b\leqslant \max\{a+\frac{1+\theta}{\alpha\gamma e^{r(T-t)}},\VaR_{\alpha}(Y)\}$.
\end{proposition}

\begin{remark}The optimal indemnity function for $c = (1+\theta)\ES_\alpha$ (with $\alpha \in (0,1)$) has a similar structure to that for $c = (1+\theta)\VaR_\alpha$, but with a more complex choice of the deductible parameter  $b^*_t$, beyond which the insurer   transfers all claims exceeding a deductible $b^*_t-a^*_t$.\end{remark}
\section{Numerical illustrations}\label{sec:5}
In this section, we investigate the impact of belief heterogeneity, where the insurer and reinsurer hold different beliefs about the parameters of the same underlying claim distribution. This discrepancy can arise when both parties assume the claim $Y$
  follows the same type of distribution, but with distinct parameter estimates based on their respective information or perspectives. 
  
  For simplicity, we assume that the claim amount $Y$ follows an exponential distribution, a widely used model in risk analysis due to its simplicity and ability to capture the general characteristics of insurance claims. The probability density functions  for the insurer and the reinsurer are given by:
\begin{align*}
	f(y)=\dfrac{1}{\theta_1}e^{-\frac{y}{\theta_1}},~~\text{ and }f^{\mathbb Q}(y)=\dfrac{1}{\theta_2}e^{-\frac{y}{\theta_2}},
\end{align*}
where $\theta_1>0$
  and $\theta_2>0$
  represent the scale parameters for the insurer and the reinsurer, respectively. A larger $\theta_i$ indicates a higher average claim size.   The likelihood ratio between the two distributions is given by: \begin{equation}\label{LR:exp}
      \mathrm{LR}(y)=\dfrac{\theta_1}{\theta_2}e^{(\frac{1}{\theta_1}-\frac{1}{\theta_2})y},\end{equation} and  its derivative is $$\mathrm{LR}'(y)=\dfrac{\theta_2-\theta_1}{\theta_2^2}e^{(\frac{1}{\theta_1}-\frac{1}{\theta_2})y}.$$ Next, we summarize the optimal reinsurance strategy based on the relationship between  $\theta_1$ and $\theta_2$.
\begin{proposition}\label{Prop:6}Assume that  the likelihood
ratio is given by \eqref{LR:exp}.   For Problem \ref{pro:1}, we have the following: 
\begin{itemize}
\item[(i)] If $\frac{\theta_1}{\theta_2}\geq 1$, then the optimal indemnity function is $I^*(t, y)=(y-d^*_t)_+,$ where \begin{equation}\label{eq:d_ex1}d^*_t=\inf\left\{d\geq 0: 1+\gamma e^{r(T-t)}d-(1+\theta)e^{-(\frac{1}{\theta_2}-\frac{1}{\theta_1})d}\geq0\right\}.\end{equation}
\item[(ii)] If $\frac{\theta_1}{\theta_2}\leq1-\gamma e^{r(T-t)}\frac{\theta_2} {1+\theta}$, then the optimal indemnity function is  $I^*(t,y)=x\wedge d^*_t$, where 
\begin{equation}\label{eq:d_ex2}d^*_t=\left\{\begin{aligned}&0,   ~~&\theta \geq \gamma\theta_1 e^{r(T-t)},\\ &\frac{\theta_1\theta_2}{\theta_2-\theta_1}\ln\frac{1+\gamma e^{r(T-t)}\theta_1}{1+\theta}, &  \theta < \gamma\theta_1 e^{r(T-t)}.\end{aligned}\right.\end{equation}
\item[(iii)] If $1-\gamma e^{r(T-t)}\frac{\theta_2} {1+\theta} < \frac{\theta_1}{\theta_2}<1$, then the
optimal form of the indemnity function  is given by 
\begin{align}\label{eq:I_abl}
I^*(t, y)=\min\{\max\{\phi_{\lambda^*_t}\chi_{\{y\leqslant y_1\}},a^*_t+y-y_1,0\},y,a^*_t\}+(y-y_1)_+-(y-d^*_t)_+,
\end{align}
where  \begin{equation}\label{eq:phi}
    \phi_{\lambda}(t,y)=y-\frac{(1+\theta)\theta_1}{\theta_2\gamma e^{r(T-t)}}e^{(\frac{1}{\theta_1}-\frac{1}{\theta_2})y}+\frac{\lambda}{\gamma e^{r(T-t)}},\end{equation}
and 
$$y_1=\frac{\theta_1\theta_2}{\theta_2-\theta_1}r(T-t)+\frac{\theta_1\theta_2}{\theta_2-\theta_1}\ln\frac{\gamma\theta_2^2}{(1+\theta)(\theta_2-\theta_1)}.$$ Moreover, let $I_{a,d,\lambda}(t,y)=\min\{\max\{\phi_{\lambda}\chi_{\{y\leqslant y_1\}},a+y-y_1,0\},y,a\}+(y-y_1)_+-(y-d)_+$, then we have $$(a^*_t,d^*_t,\lambda^*_t)=\argmin_{(a,d,\lambda)\in[0,y_1]\times\in[y_1,\infty)\times\mathbb R} H(t,I^*_{a,d,\lambda}(t,Y)),$$
where  \begin{equation}\label{eq:H1}H(t,I^*_{a,d,\lambda}(t,Y))=(1+\theta)\mathbb E[I^*_{a,d,\lambda}(t,Y)\mathrm{LR}(Y)]+\mathbb E[Y-I^*_{a,d,\lambda}(t,Y)]+\dfrac{\gamma e^{r(T-t)}}{2}\mathbb E[(Y-I^*_{a,d,\lambda}(t,Y))^2].\end{equation}
\end{itemize}
\end{proposition}

\begin{remark} The relationship between $I^*$ in equation  \eqref{eq:I_abl}  and the parameters $a$, $d$, and $\lambda$ is illustrated in Figure \ref{fig:expr}.  Note that $\phi_{\lambda}(t,\cdot)$ is an increasing concave function on $[0,y_1]$ with $\frac{\partial \phi_{\lambda}(t,y)}{\partial y}\in[0,1]$. As mentioned below Theorem \ref{form},  for any fixed values of $a$ and $d$, the value of $I^*$ remains unchanged when $$\lambda\leqslant \gamma e^{r(T-t)}(a-y_1)+(1+\theta)\frac{\theta_1}{\theta_2}e^{(\frac{1}{\theta_1}-\frac{1}{\theta_2})(y_1-a)}=:\lambda_0,$$ or $$\lambda\geqslant (1+\theta)\frac{\theta_1}{\theta_2}e^{(\frac{1}{\theta_1}-\frac{1}{\theta_2})a}=:\lambda_1.$$ Thus, we can further narrow down the range of $\lambda$ to $[\lambda_0,\lambda_1]$.
\begin{figure}[!htbp]
\centering
\includegraphics[scale=0.48]{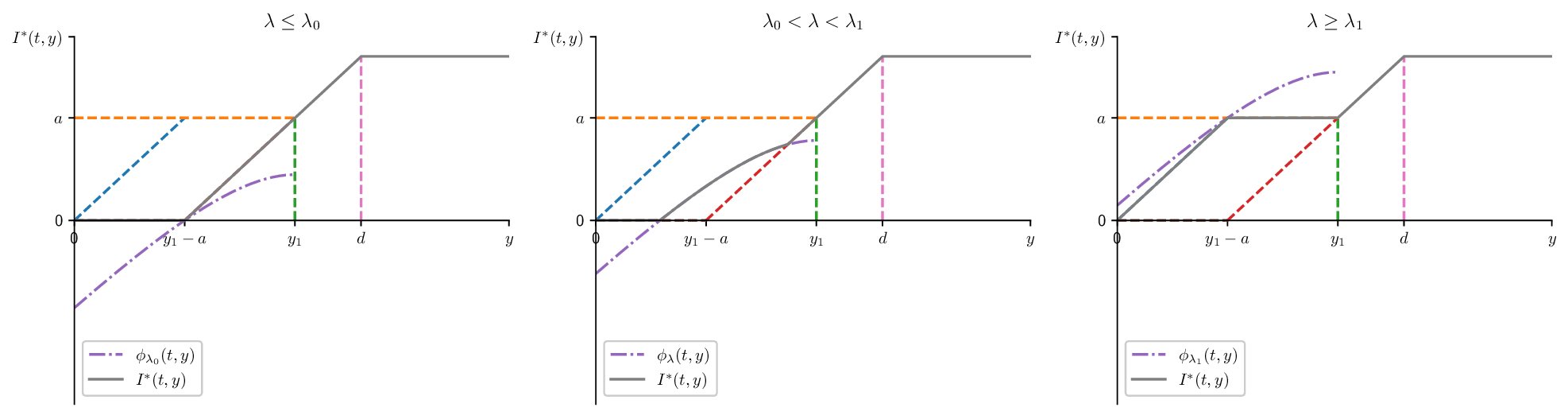}
\caption{The relationship between $I^*$ and the parameters $a$, $d$, and $\lambda$}
\label{fig:expr}
\end{figure}
\end{remark}

The reinsurance strategy described in Proposition \ref{Prop:6} is   sensitive to the relationship between the
scale parameters $\theta_1$ and $\theta_2$.  
 \begin{itemize}
 \item  In Case (i), where, $\frac{\theta_1}{\theta_2} \geq 1$, the optimal indemnity function  takes the form of an excess-of-loss structure, with the insurer opting to transfer all claims exceeding a deductible $d^*_t$. A numerical illustration of this is provided in the left panel of Figure \ref{fig:exp}. This result is intuitive, as the insurer perceives claim severity to be higher than the reinsurer's assessment. Moreover, we can see from \eqref{eq:d_ex1}  that  $d^*_t$  depends on both $\theta_1$ and $\theta_2$, specifically,  decreasing  with $\theta_1$ and increasing  with $\theta_2$, which aligns with expectations.    

Additionally, we observe that the deductible  $d^*_t$  decreases with   $\gamma$ and $r$, but increases with $\theta$ and $t$. These dynamics are consistent with the reasoning presented in Remark  \ref{rem:2}.  
A lager $\gamma$ reflects a greater risk aversion on the part of the insurer, which drives the insurer to transfer more risk to the reinsurer. 
A higher $\theta$ leads to higher premiums, incentivizing the insurer to retain a larger share of the claims. Furthermore, {a higher 
$r$  reduces the insurer's willingness to accept larger risks, as he prefers to allocate more capital to investments rather than to claims.}    Finally, as time 
$t$ approaches the maturity date 
$T$, the insurer's confidence increases in the likelihood of large claims materializing in the short term, prompting  the insurer to retain more risk. These explanations also apply to the following two cases.

 \item In Case (ii), where   $\frac{\theta_1}{\theta_2} \leq 1 - \gamma e^{r(T-t)}\frac{\theta_2}{1+\theta}$, 
 the optimal indemnity function follows a limited-loss structure, as shown in the middle panel of Figure  \ref{fig:exp}. This observation is also intuitively grounded.  The insurer perceives the claims to be less severe, while the reinsurer, believing the claims could be more severe.
  Consequently, the insurer retains all claims above a fixed threshold $d^*_t$ as given by \eqref{eq:d_ex2}. Notably, when  $\theta>\gamma\theta_1e^{r(T-t)}$, the insurer  chooses not to purchase any reinsurance, opting to retain the entire risk.

\item In case (iii), where $1 - \gamma e^{r(T-t)}\frac{\theta_2}{1+\theta} < \frac{\theta_1}{\theta_2} < 1$,
  the insurer perceives the claims to be less severe than the reinsurer, but the difference is not large enough to make the insurer's decision fully dominated by the reinsurer's premium. As a result, the optimal indemnity function adopts a more complex piecewise structure, reflecting a nuanced allocation of risk between the insurer and reinsurer, balancing their differing assessments of claim severity.   A numerical  example of this is provided in the right panel of Figure \ref{fig:exp}.
\end{itemize}

 The numerical solutions for the three cases described above are shown in Figure \ref{fig:exp}.  For our numerical experiments, we set the following parameters:
\begin{itemize}
\item Case (i):  $\theta_1=2$, $\theta_2=1$, $\gamma=1$, $r=0.1$, $T=10$, and $\theta=0.35$. 
\item Case (ii):  $\theta_1=0.5$, $\theta_2=1$, $\gamma=0.1$, $r=0.1$, $T=10$, and $\theta=0.05$. 
\item Case (iii):  $\theta_1=1.5$, $\theta_2=2$, $\gamma=0.5$, $r=0.1$, $T=10$, and $\theta=0.35$.
\end{itemize}

\begin{figure}[!htbp]
\centering
\includegraphics[width=\textwidth]{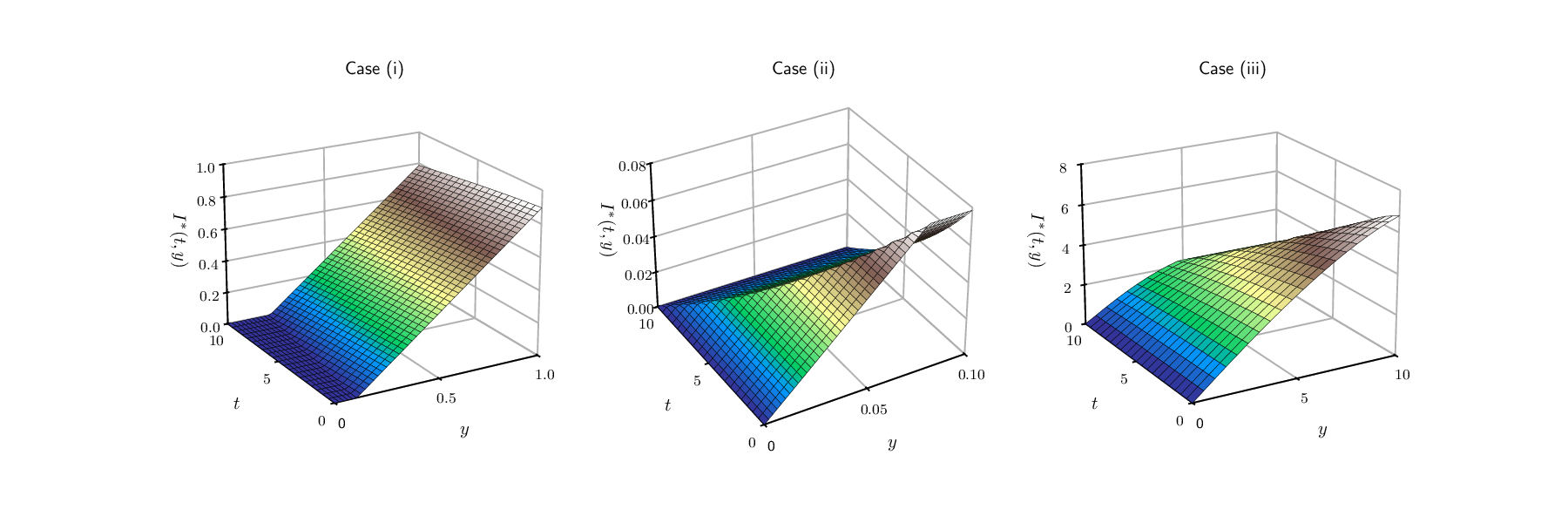}
\caption{The   optimal indemnity function $I^*$ under different parameter settings}
\label{fig:exp}
\end{figure}

Recall from Corollary \ref{coro:1} that,  if $\mathbb P=\mathbb Q$, the optimal equilibrium strategy is given by   $ \overline I^*(t, y)=(x-\overline d^*_t)_+,$ where $\overline d^*_t=\frac{\theta} {\gamma e^{r(T-t)}}.$
In addition, by Corollary \ref{coro:2}, 
 the  optimal indemnity function, $\widetilde I^*$,   in the absence of the incentive compatibility constraint, can be solved numerically by   $$\widetilde{I}^*(t,y)= \min\{y,\max\{0,\phi_{\lambda^*_t}(t,y)\}\},$$  where  $\phi_{\lambda}(t,y)$ is defined by \eqref{eq:phi}. The parameter $\lambda^*_t$ can be solved by  
 $\lambda^*_t=\argmin\limits_{\lambda\in\mathbb R} H(t,I_{\lambda}(t,Y))$, where  
$H$ is given by \eqref{eq:H1} and $I_{\lambda}(t,Y)=\min\{y,\max\{0,\phi_{\lambda}(t,y)\}\}$.

Next, we compare our strategy with the two special cases. For this purpose, we set the parameters as above, and plot the corresponding strategies as functions of $y$. In particular, we fix $t=5$.   
\begin{figure}[!htbp]
\centering
\includegraphics[scale=0.5]{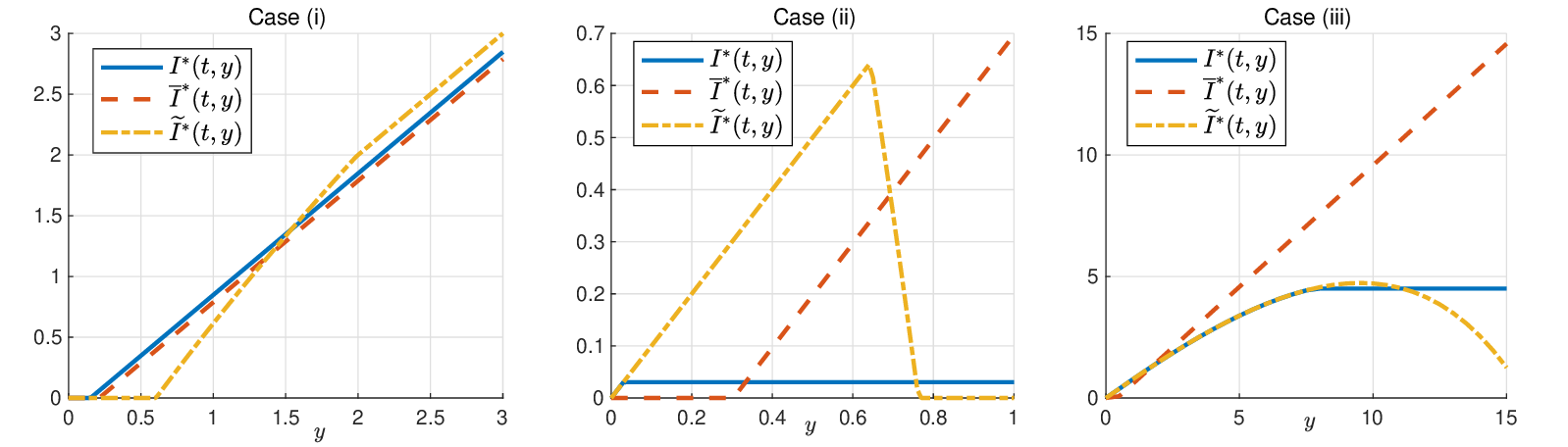}
\caption{Comparison of different strategies with varying 
$y$-values}
\label{fig:exp2}
\end{figure}

We observe from the left panel of Figure \ref{fig:exp2} that both  $I^*$ and $\overline I^*$ exhibit an excess-of-loss structure. Since the insurer perceives claim severity to be higher than the reinsurer's assessment, it is expected that the insurer will transfer more risk to the reinsurer compared to the belief heterogeneity case. Moreover, we observe that the rate of increase of  $\widetilde I^*$  can exceed 1, which suggests the presence of moral hazard. Without any binding constraints, the insurer is more inclined to retain all risk in the case of small risks, while transferring more risk in the case of large risks.

{In the middle and right panels of Figure \ref{fig:exp2}, $ \widetilde{I}^*$ initially increases and then decreases as $y$ rises. This behavior arises because the insurer perceives the claims as less severe than the reinsurer does. Without moral hazard constraints, the insurer has more flexibility in balancing the trade-off between risk and premium.  As  $ y$ increases, the insurer initially opts to transfer more risk to the reinsurer to limit exposure to larger claims. However, as the cost of reinsurance rises, the insurer becomes less willing to purchase additional coverage, ultimately reducing or eliminating reinsurance purchases.}

\section{Conclusion}\label{sec:6}
This paper makes a valuable contribution to the literature by being the first to investigate continuous-time mean-variance reinsurance design under belief heterogeneity, thereby enhancing the understanding of reinsurance contracts in real-world settings. 
Unlike much of the existing literature, where moral hazard is typically mitigated through predefined contract structures, we show that introducing heterogeneous beliefs complicates the reinsurance design and leads to the emergence of moral hazard, which requires careful consideration of an incentive compatibility constraint.

Our key findings are as follows:
(1) The optimal contracts exhibit more complex structures than the standard proportional and excess-of-loss reinsurance commonly studied in the literature, incorporating partial reinsurance across multiple layers (Theorem \ref{form}).
(2) Under mild assumptions, we establish that the equilibrium strategy is both unique and continuous over time (Theorem \ref{thm3}).
(3) When belief heterogeneity is described by the distortion function, we find that the optimal indemnity function takes the form of excess-of-loss for convex distortion functions, and adopts a dual truncated excess-of-loss form when the premium is calculated by VaR and ES (Propositions \ref{cov-dis}-\ref{prop:5}).
(4) Through numerical comparisons, we demonstrate that, compared to models that omit belief heterogeneity, our model better captures the insurer's decision-making process under various risk scenarios  (Section \ref{sec:5}).  Furthermore, since our model inherently avoids moral hazard, it provides a more reasonable and realistic framework for reinsurance design.
%In conclusion, our work offers a more detailed and realistic framework for reinsurance design by addressing belief heterogeneity and moral hazard, providing valuable insights for both theoretical advancements and practical applications in insurance risk management.

\vspace{1cm}
\noindent
{\large \bf Disclosure statement}

\vspace{0.2cm}
\noindent
 No potential conflict of interest was reported by the authors.
 
 \noindent
 
 \vspace{1cm}
\noindent{\large {\bf Acknowledgments}}
 \vspace{0.3cm}
\noindent

The  research of  Junyi Guo is supported by the National Natural Science Foundation of China (Grant No  12271274). The  research of  Xia Han  is supported by the National Natural Science Foundation of China (Grant Nos. 12301604, 12371471, and 12471449).    

\section*{Appendix}
\label{appendix}
\begin{proof}[Proof of Proposition \ref{sol_exist}]
	Let $m$ be the infimum of \eqref{eq:H}, and  let $\{I_n\}_{n=1,2,...}\subset\mathcal C$ be a sequence such that, for each  $n$, 
	\begin{align*}
		(1+\theta)\mathbb E^{\mathbb Q}[I_n(Y)]+\mathbb E[Y-I_n(Y)]+\dfrac{\gamma e^{r(T-t)}}{2}\mathbb E[(Y-I_n(Y))^2]<m+\dfrac{1}{n}.
	\end{align*}
	By the definition of $\mathcal C$, we know that on any compact interval $[0,M]$, $\{I_n\}$ is equicontinuous and uniformly bounded by $M$. Thus, according to the Arzelà-Ascoli theorem, the sequence $\{I_n\}$ is precompact in $C([0,M])$ and there exists a subsequence of $\{I_n\}$ that converges uniformly on $[0,M]$. Next we use the diagonal argument to find a point-wise convergent sequence. 
	
	For each integer $k$, consider the compact interval $[0, k]$:
\begin{itemize}
    \item By the Arzelà-Ascoli theorem, there exists a subsequence $\{I^{(1)}_n\} \subset \{I_n\}$ that converges uniformly on $[0, 1]$ to some function $I^{(1)}$.
    \item From $\{I^{(1)}_n\}$, we extract a further subsequence $\{I^{(2)}_n\} \subset \{I^{(1)}_n\}$ that converges uniformly on $[0, 2]$ to some function $I^{(2)}$.
    \item Repeat this process for each $k$, producing a nested sequence of subsequences:
    \[
    \{I^{(1)}_n\} \supset \{I^{(2)}_n\} \supset \{I^{(3)}_n\} \supset \cdots,
    \]
    where $\{I^{(k)}_n\}$ converges uniformly on $[0, k]$ to a function $I^{(k)}$.
\end{itemize}
Let $\{I_{n_j}\}$ be the diagonal subsequence, that is, $I_{n_j}=I^{(j)}_j$. Then $\{I_{n_j}\}$ converges uniformly on every compact interval $[0,k]$, and thus $\{I_{n_j}\}$ converges pointwise to a continuous function $I^*\in C([0,\infty))$. Further, we have 
\begin{align*}
	I^*(0)=\lim\limits_{j\rightarrow\infty}I_{n_j}(0)=0,
\end{align*}
and for any $x\leqslant y$,
\begin{align*}
	0\leqslant I^*(y)-I^*(x) = \lim\limits_{j\rightarrow\infty}I_{n_j}(y)-I_{n_j}(x)\leqslant y-x.
\end{align*}
Then $I^*\in\mathcal C$. By the dominated convergence theorem, $I^*$ realizes the infimum.

We now prove the uniqueness of the solution.  Assume, for the sake of contradiction, that  $I_1$ and $ I_2$  are  distinct solutions that both attain the infimum of  \eqref{eq:H}.  Consider the convex combination of $ I_1 $ and $ I_2 $ defined as
$$
\widetilde I_{\lambda}(x) = \lambda I_1 (x)+ (1 - \lambda) I_2(x), \lambda \in [0, 1].
$$ The we have
$$\begin{aligned}
H(t, \widetilde I_{\lambda})& =(1+\theta)\mathbb E^{\mathbb Q}[\widetilde I_{\lambda}(Y)]+\mathbb E[Y-\widetilde I_{\lambda}(Y)]+\dfrac{\gamma e^{r(T-t)}}{2}\mathbb E[(Y-\widetilde I_{\lambda}(Y))^2]\\&<\lambda( (1+\theta)\mathbb E^{\mathbb Q}[ I_1(Y)]+\mathbb E[Y-I_1(Y)]+\dfrac{\gamma e^{r(T-t)}}{2}\mathbb E[(Y-I_1(Y))^2])\\&~+(1-\lambda)( (1+\theta)\mathbb E^{\mathbb Q}[ I_2(Y)]+\mathbb E[Y- I_2(Y)]+\dfrac{\gamma e^{r(T-t)}}{2}\mathbb E[(Y-I_2(Y))^2])\\&=H(t,I_1)=H(t,I_2),  
\end{aligned}$$ which leads to a contradiction. The strict inequality follows from Jessen's inequality. Therefore, we conclude that  $I_1(Y)=I_2(Y)$ almost surely. 

\end{proof}

\begin{proof}[Proof of Lemma \ref{L}]
Since  $I(y)=\int_0^yI'(s)\d s$, we have 
\begin{align*}
\begin{split}
	&\gamma e^{r(T-t)}\mathbb E[I^*(t,Y)I(Y)]	-\gamma e^{r(T-t)}\mathbb E[YI(Y)]+(1+\theta)\mathbb E^{\mathbb Q}[I(Y)]-\lambda\mathbb E[I(Y)]\\
	=&\int_0^{\infty}(\gamma e^{r(T-t)}I^*(t,y)-\gamma e^{r(T-t)}y-\lambda)I(y)\d F(y)+(1+\theta)\int_0^{\infty}I(y)\d F^{\mathbb Q}(y)\\
	=&\int_0^{\infty}(\gamma e^{r(T-t)}I^*(t,y)-\gamma e^{r(T-t)}y-\lambda)\int_0^yI'(s)\d s\d F(y)+(1+\theta)\int_0^{\infty}\int_0^yI'(s)\d s\d F^{\mathbb Q}(y)\\
	=&\int_0^{\infty}\int_s^{\infty}(\gamma e^{r(T-t)}I^*(t,y)-\gamma e^{r(T-t)}y-\lambda)\d F(y)I'(s)\d s+(1+\theta)\int_0^{\infty}\int_s^{\infty}\d F^{\mathbb Q}I'(s)\d s\\
	=&\int_0^{\infty}L(s; I^*,\lambda)I'(s)\d s.
\end{split}
\end{align*}
Therefore, the optimization problem \eqref{argmin} can be written as
\begin{align*}
	I^*\in \argmin_{I\in\mathcal C}\int_0^{\infty}L(s; I^*,\lambda)I'(s)ds.
\end{align*}
To minimize this integral, a necessary and sufficient condition is that
\begin{align*}
	I'(s)=\left\{
	\begin{aligned}
		&1,\ \ &\text{if }L(s;  I^*,\lambda)<0,\\
		&0,\ \ &\text{if }L(s;  I^*,\lambda)>0,\\
		&\xi(s),&\text{if }L(s;  I^*,\lambda)=0,
	\end{aligned}
	 \right.
\end{align*}
where $\xi(s)$  could be any $[0,1]$-valued Lebesgue-measurable function  such that $I^* \in \mathcal{C}$. This condition corresponds to equation \eqref{dI}, completing the proof.
\end{proof}

\begin{proof}[Proof of Theorem \ref{form}]

{We first emphasize that, due to the incentive compatibility condition, we restrict our search for  strategies to the class $\mathcal{C}= \{f:  \R_+\to\R_+   \mid  f(0)=0, 0\leqslant f(y)-f(x)\leqslant y-x, \forall x\leqslant y\}$, where the strategies are continuous. Consequently, the resulting strategy $I^*$ is continuous with respect to $y$. Once the values of $I^*(t,\cdot)$ on the interior of each interval are known, its values at the endpoints can be naturally determined by continuity.

%Moreover, for each endpoint $y_{i-1}$, there exists a small sub-interval $[y_{i-1}, y_{i-1}'] \subset [y_{i-1}, y_i]$ such that $I^*(t,\cdot)$ is continuous on $[y_{i-1}, y_{i-1}']$ and differentiable on the open set $(y_{i-1}, y_{i-1}')$. By the Lagrange Mean Value Theorem, we have $$\frac{I^*(t,y_{i-1}') - I^*(t,y_{i-1})}{y_{i-1}' - y_{i-1}} = \frac{\partial I^*}{\partial y}(t,y_{\theta}), \quad y_{\theta} \in (y_{i-1}, y_{i-1}').$$ As long as $\frac{\partial I^*}{\partial y}(t,y_{\theta})$ admits a limit as $y_{i-1}' \rightarrow y_{i-1}$, we can  define the right derivative of $I^*(t,\cdot)$ at $y_{i-1}$; a similar argument applies to the left derivative.
}

Now we  will apply Lemma \ref{L} to complete the proof. To begin,  we have that on every sub-interval
\begin{align*}
	L'(s;I^*,\lambda)=(\lambda+\gamma e^{r(T-t)}s-\gamma e^{r(T-t)}I^*(t,s)-(1+\theta)\text{LR}(s))f(s).
\end{align*}
Let
\begin{align*}
	K(s):=\lambda+\gamma e^{r(T-t)}s-\gamma e^{r(T-t)}I^*(t,s)-(1+\theta)\text{LR}(s),
\end{align*}
so that 
\begin{align*}
	K'(s)=\gamma e^{r(T-t)}-\gamma e^{r(T-t)}\frac{\partial I^*}{\partial y}(t,s)-(1+\theta)\text{LR}'(s).
\end{align*}
Note that $K(s)$ and $L'(s;I^*,\lambda)$ have the same sign.  Thus,  when considering  the sign of $L'(s;I^*,\lambda)$, it suffices to focus on the sign of 
  $K(s)$.\\

\noindent (i) If $\frac{\partial \phi_{\lambda}}{\partial y}(t,y)>1$ on ${(}y_{m_t},\infty)$, then $\text{LR}'(y)<0$. Consequently,  $K'(s)>0$, implying that $K(s)$ is a {strictly} increasing function.  Thus,  $K(s)=0$ can hold at most at one  point. This shows that   $L(s;I^*,\lambda)=0$ is impossible on any sub-interval of ${(}y_{m_t},\infty)$. 

Next we prove that $L(s;I^*,\lambda)$ cannot up-cross the $s$-axis on ${(}y_{m_t},\infty)$. If this were not true, there would exist a point $s^*$ such that
\begin{align*}
	L(s^*;I^*,\lambda)>0,\ L'(s^*;I^*,\lambda)\geqslant 0.
\end{align*}
This implies $K(s^*)\geqslant 0$, and since $K(s)$ is {strictly} increasing, we  have  $K(s)\geqslant 0$ for all $[s^*,\infty)$ . Consequently, $L'(s;I^*,\lambda)\geqslant 0$ on $[s^*,\infty)$ and 
\begin{align*}
	L(s^*;I^*,\lambda)=\int_{s^*}^{\infty}-L'(s;I^*,\lambda)\d s\leqslant 0,
\end{align*}
which contradicts the assumption that  $L(s^*;I^*,\lambda)>0$.

Let $s_0:=\inf\{s\in{(}y_{m_t},\infty)|L(s;I^*,\lambda)\leqslant 0\}$. Then $L(s;I^*,\lambda)>0$ on ${(}y_{m_t},s_0)$ and $L(s;I^*,\lambda)\leqslant 0$ on $[s_0,\infty)$. By Lemma \ref{L}, we know that $\frac{\partial I^*}{\partial y}(t,y)=\chi_{[s_0,\infty)}(y)$.\\

\noindent (ii) If $\frac{\partial \phi_{\lambda}}{\partial y}(t,y)\in[0,1]$ on ${(}y_{m_t},\infty)$, then $0\leq \text{LR}'(y)\leq \gamma e^{r(T-t)}/(1+\theta)$. We now show that $L(s;I^*,\lambda)$ cannot up-cross or  down-cross the $s$-axis.

Assume, for the sake of contradiction, that $L(s;I^*,\lambda)$  up-crosses the $s$-axis. Then there must exist a point $s^*\in{(}y_{m_t},\infty)$ such that 
\begin{align*}
	L(s^*;I^*,\lambda)>0,\ L'(s^*;I^*,\lambda)\geqslant 0.
\end{align*}
From this, we deduce that  $K(s^*)\geqslant 0$. By Lemma \ref{L} and the condition $L(s^*;I^*,\lambda)>0$, we know that $\frac{\partial I^*}{\partial y}(t,s^*)=0$. Consequently, $$K'(s^*)=\gamma e^{r(T-t)}-(1+\theta)\text{LR}'(s^*)\geqslant 0.$$ This implies $K(s)\geqslant 0$ and $L'(s^*;I^*,\lambda)\geqslant 0$ on $[s^*,\infty)$. Further, we have
\begin{align*}
	L(s^*;I^*,\lambda)=\int_{s^*}^{\infty}-L'(s;I^*,\lambda)\d s\leqslant 0, 
\end{align*}
which is a contradiction with $L(s^*;I^*,\lambda)>0$. Thus, $L(s;I^*,\lambda)$ cannot up-cross the $s$-axis. 

A similar argument can be made to show that $L(s;I^*,\lambda)$ cannot down-cross the $s$-axis. Thus, we conclude that if there exists a  point $\widetilde s$ such that $L(\widetilde s;I^*,\lambda)=0$, then $L(s;I^*,\lambda)=0$ for all $[\widetilde s,\infty)$.

On $[\widetilde s,\infty)$, we have $L'(s^*;I^*,\lambda)=0$, which implies that  $K(s)=0$, and hence $I^*(t,s)=\phi_{\lambda}(t,s)$. Let $s_0:=\inf\{s\in{(}y_{m_t},\infty)|L(s;I^*,\lambda)= 0\}$. Then $L(s;I^*,\lambda)\lessgtr 0$ on ${(}y_{m_t},s_0)$ and $L(s;I^*,\lambda)= 0$ on $[s_0,\infty)$.  More precisely, we have 
\begin{align*}
	\frac{\partial I^*}{\partial y}(t,y)=\left\{
	\begin{aligned}
		&\frac{\partial \phi_{\lambda}}{\partial y}(t,y)\chi_{[s_0,\infty)}(y)&\text{ if }L(s;I^*,\lambda)>0 \text{ on } {(}y_{m_t},s_0),\\
		&\chi_{{(}y_{m_t},s_0)}(y)+\frac{\partial \phi_{\lambda}}{\partial y}(t,y)\chi_{[s_0,\infty)}(y)&\text{ if }L(s;I^*,\lambda)<0 \text{ on } {(}y_{m_t},s_0).
	\end{aligned}
	\right.
\end{align*}

\noindent (iii) If $\frac{\partial \phi_{\lambda}}{\partial y}(t,y)<0$ on ${(}y_{m_t},\infty)$, then  and $\text{LR}'(y)> \gamma e^{r(T-t)}/(1+\theta)$. Similar to the previous case (i), we have $K'(s)<0$ on ${(}y_{m_t},\infty)$ and it is impossible that $L(s;I^*,\lambda)=0$ on any sub-interval of ${(}y_{m_t},\infty)$. 

Next, we prove that $L(s;I^*,\lambda)$ cannot down-cross the $s$-axis on ${(}y_{m_t},\infty)$. Suppose this is not true, then there must exist some point $s^*\in{(}y_{m_t},\infty)$ such that
\begin{align*}
	L(s^*;I^*,\lambda)<0,\ L'(s^*;I^*,\lambda)\leqslant 0.
\end{align*}
From this, we deduce that  $K(s^*)\leqslant 0$. Since $K'(s)<0$, we know that $L'(s^*;I^*,\lambda)\leqslant 0$ on $[s^*,\infty)$,  and therefore, 
\begin{align*}
	L(s^*;I^*,\lambda)=\int_{s^*}^{\infty}-L'(s;I^*,\lambda)\d s\geqslant 0,
\end{align*}
which contradicts the assumption that $L(s^*;I^*,\lambda)<0$. 

Let $s_1:=\inf\{s\in{(}y_{m_t},\infty)|L(s;I^*,\lambda)\geqslant 0\}$. Then $L(s;I^*,\lambda)<0$ on ${(}y_{m_t},s_1)$ and $L(s;I^*,\lambda)\geqslant 0$ on $[s_1,\infty)$. By Lemma \ref{L}, we conclude  that $\frac{\partial I^*}{\partial y}(t,y)=\chi_{{(}y_{m_t},s_1)}(y)$. \\

\noindent (iv) If $\frac{\partial \phi_{\lambda}}{\partial y}(t,y)>1$ on ${(}y_{i-1},y_i{)}$, then $K'(s)>0$ and it is impossible that $L(s;I^*,\lambda)=0$ on any sub-interval of ${(}y_{i-1},y_i{)}$. Further, $K(s)=0$ can have at most one root. Let $s_r$ denote this root, if it exists. If  $s_r\in{(}y_{i-1},y_i{)}$, we have $K(s)<0$ on ${(}y_{i-1},s_r)$ and $K(s)>0$ on $(s_r,y_i{)}$. If $s_r\notin {(}y_{i-1},y_i{)}$ or the root does not exist, then $L(s;I^*,\lambda)\lessgtr 0$ on ${(}y_{i-1},y_i{)}$. In conclusion, $L(s;I^*,\lambda)$ can cross the $s$-axis at most twice. Let
\begin{align*}
	s_{i,0}:=\inf\{s\in{(}y_{i-1},y_i{)}|L(s;I^*,\lambda)\leqslant 0\},\ s_{i,1}:=\inf\{s\in(s_{i,0},y_i{)}|L(s;I^*,\lambda)\geqslant 0\}.
\end{align*}
Then on ${(}y_{i-1},s_{i,0})$ and $(s_{i,1},y_i{)}$, we have $L(s;I^*,\lambda)>0$,  and on $(s_{i,0},s_{i,1})$, we have $L(s;I^*,\lambda)<0$. Therefore,
\begin{align*}
	\frac{\partial I^*}{\partial y}(t,y)=\chi_{[s_{i,0},s_{i,1}]}(y),
\end{align*}
which, combined with continuity of $I^*(t,\cdot)$, we derive the result.\\

\noindent (v) If $\frac{\partial \phi_{\lambda}}{\partial y}(t,y)\in[0,1]$ on ${(}y_{i-1},y_i{)}$, then $0\leq \text{LR}'(y)\leq \gamma e^{r(T-t)}/(1+\theta)$. Using the same argument in the proof of (ii), it follows that  if there exists some  point $s^*$ such that
\begin{align*}
	L(s^*;I^*,\lambda)>0,\ L'(s^*;I^*,\lambda)\geqslant 0,
\end{align*}
then $L(s;I^*,\lambda)>0$ for all $s\in [s^*,y_i{)}$. Similarly,  if there exists some $s^{**}$ such that
\begin{align*}
	L(s^{**};I^*,\lambda)<0,\ L'(s^{**};I^*,\lambda)\leqslant 0,
\end{align*}
then $L(s;I^*,\lambda)<0$ on $[s^{**},y_i{)}$. Thus, $L(s;I^*,\lambda)$ can be $0$ in at most one sub-interval $[s_{i,0},s_{i,1}]\subset{(}y_{i-1},y_i{)}$.

According to the sign of $L(s;I^*,\lambda)$ on ${(}y_{i-1},s_{i,0})$ and $(s_{i,1},y_i{)}$, and leveraging the continuity of $I^*$,  we have the following cases:
\begin{enumerate}[label=\arabic*)]
	\item If $L(s;I^*,\lambda)<0$ on ${(}y_{i-1},s_{i,0})$ and $L(s;I^*,\lambda)<0$ on $(s_{i,1},y_i{)}$, then $\frac{\partial I^*}{\partial y}(t,y)=\chi_{{(}y_{i-1},s_{i,0})}(y)+\frac{\partial \phi_{\lambda}}{\partial y}(t,y)\chi_{[s_{i,0},s_{i,1}]}(y)+\chi_{(s_{i,1},y_i{)}}(y)$.
	\item If $L(s;I^*,\lambda)<0$ on ${(}y_{i-1},s_{i,0})$ and $L(s;I^*,\lambda)>0$ on $(s_{i,1},y_i{)}$, then $\frac{\partial I^*}{\partial y}(t,y)=\chi_{{(}y_{i-1},s_{i,0})}(y)+\frac{\partial \phi_{\lambda}}{\partial y}(t,y)\chi_{[s_{i,0},s_{i,1}]}(y)$.
	\item If $L(s;I^*,\lambda)>0$ on ${(}y_{i-1},s_{i,0})$ and $L(s;I^*,\lambda)>0$ on $(s_{i,1},y_i{)}$, then $\frac{\partial I^*(t,y)}{\partial y}=\frac{\partial \phi_{\lambda}(t,y)}{\partial y}\chi_{[s_{i,0},s_{i,1}]}(y)$.
	\item If $L(s;I^*,\lambda)>0$ on ${(}y_{i-1},s_{i,0})$ and $L(s;I^*,\lambda)<0$ on $(s_{i,1},y_i{)}$, then $\frac{\partial I^*(t,y)}{\partial y}=\frac{\partial \phi_{\lambda}(t,y)}{\partial y}\chi_{[s_{i,0},s_{i,1}]}(y)\\+\chi_{(s_{i,1},y_i{)}}(y)$.
\end{enumerate}

\noindent (vi) If $\frac{\partial \phi_{\lambda}}{\partial y}(t,y)<0$ on ${(}y_{i-1},y_i{)}$, then $K'(s)<0$,  and it is impossible that $L(s;I^*,\lambda)=0$ on any sub-interval of ${(}y_{i-1},y_i{)}$. Similar to the proof of (iv), $K(s)=0$ has at most one root, denoted as  $s_r$ if it exists. If $s_r\in{(}y_{i-1},y_i{)}$, we have $K(s)>0$ on ${(}y_{i-1},s_r)$ and $K(s)<0$ on $(s_r,y_i{)}$. If $s_r\notin {(}y_{i-1},y_i{)}$ or no root  exists, then $L(s;I^*,\lambda)\lessgtr 0$ on ${(}y_{i-1},y_i{)}$. Thus, $L(s;I^*,\lambda)$ can cross the $s$-axis at most twice. Let
\begin{align*}
	s_{i,0}:=\inf\{s\in[y_{i-1},y_i]|L(s;I^*,\lambda)\geqslant 0\},\ s_{i,1}:=\inf\{s\in(s_{i,0},y_i{)}|L(s;I^*,\lambda)\leqslant 0\}.
\end{align*}
Then on ${(}y_{i-1},s_{i,0})$ and $(s_{i,1},y_i{)}$, $L(s;I^*,\lambda)<0$ and on $(s_{i,0},s_{i,1})$, $L(s;I^*,\lambda)>0$. Therefore, we obtain 
\begin{align*}
	\frac{\partial I^*}{\partial y}(t,y)=\chi_{{(}y_{i-1},s_{i,0})}(y)+\chi_{(s_{i,1},y_i{)}}(y),
\end{align*}
which completes the proof.
\end{proof}

{\begin{proof}[Proof of Proposition \ref{prop:3}] Since $\mathrm{LR}(x)$ is decreasing over $[0,\infty)$, then by Theorem \ref{form}, the optimal reinsurance strategy takes the form  $ I^*(t,y)=(y-d)_+$ for some $d\geq0.$ 
Substituting $ I^*(t,y)=(y-d)_+$  back into \eqref{inf}, we get \begin{align*} 	&c(I^*(t,\cdot))+\mathbb E[Y-I^*(t,Y)]+\dfrac{\gamma e^{r(T-t)}}{2}\mathbb E[(Y-I^*(t,Y))^2]\\ =&~(1+\theta)\int_{d}^{\infty}S^{\mathbb Q}(y)\d y+\int_0^{d}y\d F(y)+dS(d)+\dfrac{\gamma e^{r(T-t)}}{2}\left(\int_0^{d}y^2\d F(y)+d^2S(d)\right)\\
	=&: H_1(t,d). \end{align*} 
To minimize $H_1(t,d)$, we differentiate it with respect to $d$:
\begin{align*}
	\frac{\partial H_1}{\partial d}(t,d)=&S(d)-(1+\theta)S^{\mathbb Q}(d)+\gamma e^{r(T-t)}dS(d)\\
	=&S(d)\left(1+\gamma e^{r(T-t)}d-(1+\theta)\dfrac{S^{\mathbb Q}(d)}{S(d)}\right).
\end{align*}
Note that $\frac{\partial H_1}{\partial d}(t,0)=-\theta < 0$, which implies that $H_1$  is decreasing at zero.  Therefore,   the optimal threshold must satisfy  $d>0$.  
Moreover, since $H_1$ depends on $t$, so does the minimizer. We denote the optimal threshold by $d_t$ to emphasize its dependence on $t$.
\end{proof}}

\begin{proof}[Proof of Theorem \ref{thm3}]
	First, we show  the uniqueness of the equilibrium strategy. Suppose there exist two distinct  equilibrium strategy, $\bm I^1$ and $\bm I^2$.  Then   for any fixed $t\geq0$,   we have $I^1(t,y_0)-I^2(t,y_0)\neq 0$ for some $y_0 \in (0,\infty)$. Without loss of generality,  assume that $I^1(t,y_0)-I^2(t,y_0)>0$. Thus there must exist $\varepsilon >0$ such that $I^1(t,y)-I^2(t,y)>0$ for all $y\in [y_0-\varepsilon,y_0+\varepsilon]$. From Proposition \ref{sol_exist}, we know that $\mathbb P(I^1(t,Y)=I^2(t,Y))=1$,  which implies that  $\mathbb P(Y\in [y_0-\varepsilon,y_0+\varepsilon])=0$.  This  contradicts the assumption that  $F(y)$ is {strictly} increasing. Hence, the equilibrium strategy $I^*(t,y)$ is unique.
    
    Next, we prove that $I^*(t,y)$ is continuous in $t$. Unlike in Theorem \ref{form}, where the sub-intervals are fixed,  the form of the  equilibrium strategy may change as $t$ varies.  Consequently, the endpoints of these sub-intervals should be treated as parameters of the equilibrium strategy as well.   Note that some of these points are fixed  regardless of $t$. These include  $0$, the endpoints of the first type sub-intervals  where $\mathrm{LR}'(y)<0$, and the non-differentiable points of $\mathrm{LR}(y)$ and $F(y)$. Let these fixed points be denoted by  $0=x_0<x_1<x_2<...<x_p$. 
    
    We now introduce additional points to account for the maximum number of sub-intervals that may 
    arises as $t$ varies from $0$ to $T$. Specifically, on each $[x_{i-1},x_i]$ or $[x_p,\infty)$, $i=1,2,...,p$, we insert points such that the number of sub-intervals remains consistent with the partition structure outlined in Theorem \ref{form}.   For instance, if  $[x_{i-1},x_i]$ contains at most three sub-intervals, we add two points, $z_1(t)$ and $z_2(t)$, with $x_{i-1}\leqslant z_1(t)\leqslant z_2(t)\leqslant x_i$, ensuring the partition structure remains consistent. 
 This process generates a sequence of points
$0=y_0(t)\leqslant y_1(t)\leqslant y_2(t)\leqslant ...\leqslant y_q(t)$, which covers all divisions of  $[0,\infty)$ for any $t\in[0,T]$, as required  in Theorem \ref{form}.

Due to the piece-wise continuity of $\mathrm{LR}'(y)$, the sequence $y_i(t)$ is continuous in  $t$ for $i=0,1,2,...,q$. Let $I(\bm{x},t,y)$ denote  the parameterized representation of  the strategy $I(t,y)$, where  $\bm{x} \in (\mathbb R \cup \{\infty\})^n$ represents all the parameters. From Theorem \ref{form},  we know  that $I(\bm{x},t,y)$ is continuous in  $\bm{x}$. For each $t$,  define the domain of $\bm{x}$ as $D_t$. The continuity of  $y_i(t)$ implies that the Hausdorff distance (see, e.g., \cite{M00}) between  $D_t$ and $D_{t_0}$ tends to zero as  $t\rightarrow t_0$, i.e., 
$$d_H(D_t,D_{t_0})\rightarrow 0, ~ \text{as} ~~t\rightarrow t_0,$$ where  $d_H$  is   defined as
	\begin{align*}
		d_H(A,B):=\max\left\{\sup\limits_{a\in A} d(a,B), \sup\limits_{b\in B} d(A,b)\right\},
	\end{align*}
	with $d$ being the  Euclidean distance from a point to a subset  of $\R^n$.

    Substituting $I(\bm{x},t,y)$ into \eqref{eq:H},  we define the functional
	\begin{align*}
		H(\bm{x},t)=(1+\theta)\mathbb E^{\mathbb Q}[I(\bm{x},t,Y)]+\mathbb E[Y-I(\bm{x},t,Y)]+\dfrac{\gamma e^{r(T-t)}}{2}\mathbb E[(Y-I(\bm{x},t,Y))^2].
	\end{align*}
	Due to the uniqueness of the equilibrium strategy, the optimal problem $$\min\limits_{\bm{x}\in D_t}H(\bm{x},t)$$ has an unique minimizer $\bm{x}(t)$ for any $t\in [0,T]$. For any $t_0\in[0,T]$ and any $\bm{x}\in D_{t_0}$, take  sequences $\{t_n\}_{n=0}^{\infty}$ and $\{\bm{x}_n\}_{n=0}^{\infty}$ such that $t_n\rightarrow t_0$, $\bm{x}_n\in D_{t_n}$ and $\bm{x}_n\rightarrow \bm{x}$.  For each $n$, we have 
	$$ 
		H(\bm{x}(t_n),t_n)\leqslant H(\widetilde{\bm{x}},t_n),\ \forall \widetilde{\bm{x}}\in D_{t_n}.
$$
	We then take any subsequence $\{\bm{x}(t_{n_j})\}$ from $\{\bm{x}(t_n)\}$ such that $\bm{x}(t_{n_j})\rightarrow \overline{\bm{x}}\in D_{t_0}$ as $j\rightarrow \infty$. For each $j$, we also have
$$ 
		H(\bm{x}(t_{n_j}),t_{n_j})\leqslant H(\bm{x}_{n_j},t_{n_j}).
$$ 
	Let $j\rightarrow \infty$, we have 
$$
		H(\overline{\bm{x}},t_0)\leqslant H(\bm{x},t_0),
$$ 
	and thus $\overline{\bm{x}}$ is a minimizer of $\min\limits_{\bm{x}\in D_{t_0}}H(\bm{x},t_0)$ by the arbitrariness of $\bm{x}$. We conclude that  $\overline{\bm{x}}=\bm{x}(t_0)$ by the uniqueness of the minimizer and thus we have $$\lim\limits_{j\rightarrow\infty}\bm{x}(t_{n_j})=\bm{x}(t_0),$$ which proves that the equilibrium strategy is continuous in $t$.
\end{proof}

\begin{proof}[Proof of Corollary \ref{coro:1}]
Substituting $\overline I^*(t,y)=(x-d)_+$  back into \eqref{inf}, we get \begin{align*} 	&c(\overline I^*(t,\cdot))+\mathbb E[Y-\overline I^*(t,Y)]+\dfrac{\gamma e^{r(T-t)}}{2}\mathbb E[(Y-\overline I^*(t,Y))^2]\\ =&(1+\theta)\mathbb E[(Y-d)_+]+\mathbb E[Y\wedge d]+\dfrac{\gamma e^{r(T-t)}}{2}[(Y\wedge d)^2]\\ =&(1+\theta)\int_d^{\infty}S(y)\d y+\int_0^dy\d F(y)+dS(d)+\dfrac{\gamma e^{r(T-t)}}{2}\left(\int_0^dy^2\d F(y)+d^2S(d)\right)\\	=&:H(t,d). \end{align*} 
Using the first order condition, we have 
 $\frac{\partial H}{\partial d}(t,d)=\gamma e^{r(T-t)} d S(d)-\theta S(d).$  It is evident that  $\frac{\partial H}{\partial d}(t,0)=-\theta S(0)<0$ and $\lim\limits_{d\to\infty}\frac{\partial H}{\partial d}(t,d)=\infty$ since the term $\gamma e^{r(T-t)} d-\theta$ is increasing in $d$. Moreover, since $H$ depends on $t$, so does the minimizer. We denote the optimal threshold by $d_t$ to emphasize its dependence on $t$. Thus, the unique solution   is given by $\overline d^*_t={\theta} /{(\gamma e^{r(T-t)})}$. 
\end{proof}

\begin{proof}[Proof of Proposition \ref{cov-dis}] Since $g$ is convex and differentiable,  % and $  f^{\mathbb Q}(y)=g'(S(y))f^{\mathbb P}(y), $
we have that  LR$'(y)$ is decreasing on $[0,\infty)$.
 Then by Theorem \ref{form},  the equilibrium strategy has the form
$ 
	I^*(t,y)=(x-d)_+,
$ 
for some $d \geqslant 0$. 
Further, by following a similar proof to that of Proposition \ref{prob:2}, we have
\begin{align*}
	&c(I^*(t,\cdot))+\mathbb E[Y-I^*(t,Y)]+\dfrac{\gamma e^{r(T-t)}}{2}\mathbb E[(Y-I^*(t,Y))^2]\\
	=&~(1+\theta)\int_d^{\infty}S^{\mathbb Q}(y)\d y+\int_0^dy\d F(y)+dS(d)+\dfrac{\gamma e^{r(T-t)}}{2}\left(\int_0^dy^2\d F(y)+d^2S(d)\right)\\
	=&~ H_1(t,d).
\end{align*}
Next, we find the minimizer of $H_1(t,d)$. First, we compute the derivative:
\begin{align*}
	\frac{\partial H_1}{\partial d}(t,d)=&S(d)-(1+\theta)S^{\mathbb Q}(d)+\gamma e^{r(T-t)}dS(d)\\
	=&S(d)\left(1+\gamma e^{r(T-t)}d-(1+\theta)\dfrac{g(S(d))}{S(d)}\right).
\end{align*}
Let $$K(t,d)=1+\gamma e^{r(T-t)}d-(1+\theta)\dfrac{g(S(d))}{S(d)}.$$ Because  $g$ is convex, it follows that  ${g(S(d))}/{S(d)}$ is decreasing in $d$, and therefore   $K(t,d)$ is increasing in $d$. 
  Given that  $K(t,0)=-\theta<0$ and $\lim\limits_{d\to\infty} K(t,d)=\infty$, there exists a $d^*_t>0$ such that $K(t,d^*_t)=0$. Therefore, $d^*_t$ is the minimizer of $H_1(t,d)$.
\end{proof}

\begin{proof}[Proof of Proposition \ref{prop:4}]
Since Assumption \ref{ass1} does not hold, we apply the Lebesgue decomposition theorem to decompose $\mathbb Q$ into two components: $\mathbb Q^{ac}$ and $\mathbb Q^{\perp}$.  We have $\mathbb Q^{ac} \ll \mathbb P$ and $\mathbb Q^{\perp} \perp \mathbb P$.\footnote{Note that  $\mathbb{Q}^{\perp} \perp \mathbb{P}$ means that the measure $\mathbb{Q}^{\perp}$ is singular with respect to $\mathbb{P}$. This implies there exists a measurable set $B$ such that $\mathbb{Q}^{\perp}(B) = 1$ and $\mathbb{P}(B) = 0$.} Specifically, $\mathbb Q^{ac}(Y\in E) = \mathbb Q(Y\in E \cap \{\mathrm{LR} < \infty\})$ and $\mathbb Q^{\perp}(Y\in E) = \mathbb Q(Y\in E \cap \{\mathrm{LR} = \infty\}) = \mathbb Q(Y\in E \cap \{\VaR_{\alpha}(Y)\})$ for any Borel set $E$. Therefore, we can rewrite the expectation under $\mathbb Q$ in \eqref{ldp} as $\mathbb E^{\mathbb Q} = \mathbb E^{\mathbb Q^{ac}} + \mathbb E^{\mathbb Q^{\perp}}$. Upon fixing the value of $I$ on $\{\mathrm{LR} = \infty\}$, i.e., $I(\VaR_{\alpha}(Y))$, the remaining part of the problem \eqref{ldp} falls within the scope of Theorem \ref{form}.
%In this case, Assumption \ref{ass1} does not hold. By Lebesgue decomposition theorem, we can decompose $\mathbb Q$ into two parts: $\mathbb Q=\mathbb Q^{ac}+\mathbb Q^{\perp}$, where $\mathbb Q^{\perp} \perp \mathbb P $ and $\mathbb Q^{ac} \ll \mathbb P$. Among them, $\mathbb Q^{ac}(E)=\mathbb Q(E\cap \{\mathrm{LR}(y)<\infty\})$ and $\mathbb Q^{\perp}(E)=\mathbb Q(E\cap \{\mathrm{LR}(y)=\infty\})=\mathbb Q(E\cap \{\VaR_{\alpha}(Y)\})$ for any Borel set $E$. Then we can rewrite the expectation under $\mathbb Q$ in \eqref{ldp} as $\mathbb E^{\mathbb Q}=\mathbb E^{\mathbb Q^{ac}}+\mathbb E^{\mathbb Q^{\perp}}$. Upon fixing the value of $I$ on $\{LR = \infty\}$, i.e., $I(\VaR_{\alpha}(Y))$, the remaining part of the problem \eqref{ldp} falls within the scope of Theorem \ref{form}.
By Theorem \ref{form}, we can divide $[0,\infty)$ into $[0,\VaR_{\alpha}(Y))$ and $[\VaR_{\alpha}(Y),\infty)$. Then we have for $y\in [\VaR_{\alpha},\infty)$,
\begin{align}\label{v1}
I^*(t,y)=\min\left\{\max\left\{y+\frac{\lambda}{\gamma e^{r(T-t)}},I^*(t,\VaR_{\alpha}(Y))\right\},I^*(t,\VaR_{\alpha}(Y))+y-\VaR_{\alpha}(Y) \right\},	
\end{align}
and for $y\in[0,\VaR_{\alpha}(Y))$,
\begin{align}\label{v2}
	I^*(t,y)=\min\left\{\max\left\{y+\frac{\lambda}{\gamma e^{r(T-t)}},I^*(t,\VaR_{\alpha}(Y))+y-\VaR_{\alpha}(Y),0\right\},y,I^*(t,\VaR_{\alpha}(Y))\right\}.
\end{align} 
Thus $I^*(t,\cdot)$ is determined by two parameters $a=I^*(t,\VaR_{\alpha}(Y))\in[0,\VaR_{\alpha}(Y)]$ and $\lambda\in\mathbb R$. The relationship between $I^*$, $a$, and $\lambda$ is shown in Figure \ref{fig:var}.
\begin{figure}[!htbp]
\centering
\includegraphics[scale=0.48]{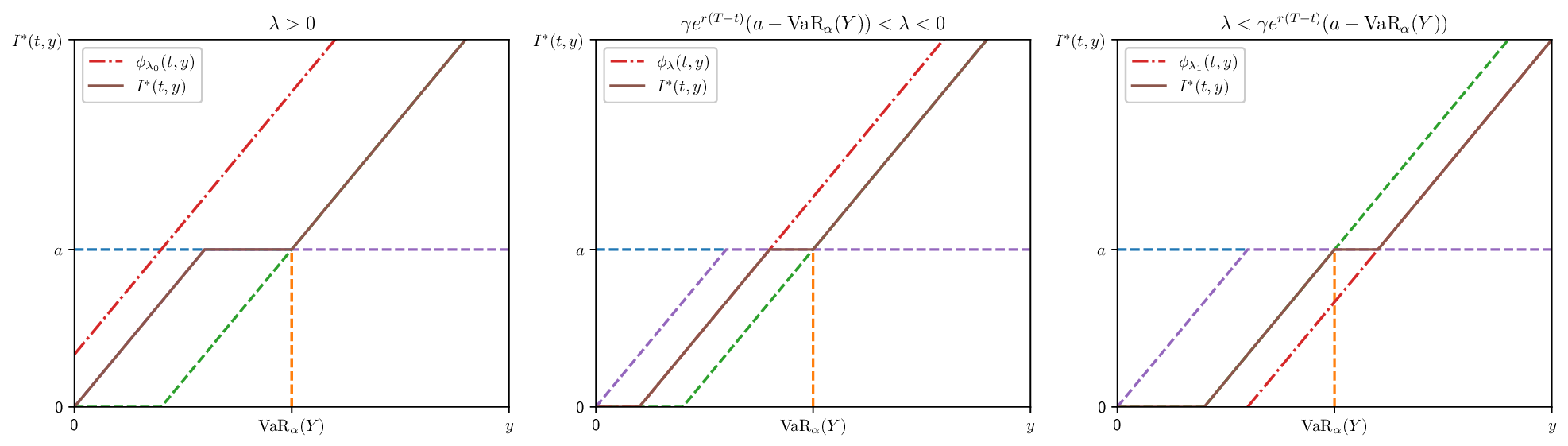}
\caption{The relationship between $I^*$ and $a$, $\lambda$}
\label{fig:var}
\end{figure}

Since
\begin{align*}
	\mathbb E^{\mathbb Q}[I^*(t,Y)]=\int_0^{\infty}I^*(t,y)\delta(y-\VaR_{\alpha}(Y))\d y=I^*(t,\VaR_{\alpha}(Y)),
\end{align*}
 it follows that
\begin{align*}
	&c(I^*(t,\cdot))+\mathbb E[Y-I^*(t,Y)]+\dfrac{\gamma e^{r(T-t)}}{2}\mathbb E[(Y-I^*(t,Y))^2]\\
	=&(1+\theta)a+\mathbb E[Y-I^*(t,Y)]+\dfrac{\gamma e^{r(T-t)}}{2}\mathbb E[(Y-I^*(t,Y))^2].
\end{align*}
To minimize \eqref{inf}, $I^*$ should be as large as possible. Thus, we only need to consider $\lambda\geqslant 0$ because $I^*$ is increasing in $\lambda$. Furthermore, $I^*$ takes the same value for all $\lambda\geqslant 0$. Thus  the equilibrium strategy can be simplified into a single parameter form
\begin{align}\label{v3}
I^*(t,y)=y\wedge a+(y-\VaR_\alpha(Y))_+.	
\end{align}
In fact, the reinsurance premium is  determined  solely by $I^*(t,\VaR_{\alpha}(Y))$, and thus with the given $I^*(t,\VaR_{\alpha}(Y))$, the insurer naturally seeks to increase $I^*(t,y)$  on  both $[0,\VaR_{\alpha}(Y))$ and $(\VaR_{\alpha}(Y),\infty)$. This is precisely the form given by \eqref{v3}. Then  it follows that 
\begin{align*}
	&c(I^*(t,\cdot))+\mathbb E[Y-I^*(t,Y)]+\dfrac{\gamma e^{r(T-t)}}{2}\mathbb E[(Y-I^*(t,Y))^2]\\
	=&~(1+\theta)a+\int_a^{\VaR_{\alpha}(Y)}(y-a)\d F(y)+\int_{\VaR_{\alpha}(Y)}^{\infty}(\VaR_{\alpha}(Y)-a)\d F(y)\\
	&+\dfrac{\gamma e^{r(t-t)}}{2}\left(\int_0^{\VaR_{\alpha}(Y)}(y-a)^2\d F(y)+\int_{\VaR_{\alpha}(Y)}^{\infty}(\VaR_{\alpha}(Y)-a)^2\d F(y)\right)\\
	=&:H_2(t,a).
\end{align*}
To find the minimizer of \eqref{inf},
we have  \begin{align*}
	\frac{\partial H_2}{\partial a}(t,a)=\theta+F(a)+\gamma e^{r(T-t)}\left(a-\int_a^{\VaR_{\alpha}(Y)}y\d F(y)-aF(a)-\alpha\VaR_{\alpha}(Y)\right).
\end{align*}
Note that $$\frac{\partial H_2}{\partial a}(t,\VaR_{\alpha}(Y))=\theta+1-\alpha>0,$$ and $$\frac{\partial H_2}{\partial a}(t,0)=\theta-\gamma e^{r(t-t)}(\mathbb E[Y]+\alpha\VaR_{\alpha}(Y)-\alpha\ES_\alpha(Y)).$$ Moreover,  we have $$\frac{\partial^2 H_2}{\partial a^2}(t,a)=f(a)+\gamma e^{r(T-t)}(1-F(a))>0,$$ which implies that  $\frac{\partial H_2}{\partial a}(t,\cdot)$ is {strictly} increasing on $[0,\VaR_{\alpha}(Y)]$. 

When $\theta<\gamma e^{r(T-t)}(\mathbb E[Y]+\alpha\VaR_{\alpha}(Y)-\alpha\ES_{\alpha}(Y))$, we have $\frac{\partial H_2}{\partial a}(t,0)<0$, and there exists a unique $a^*_t\in[0,\VaR_{\alpha}(Y)]$ such that $\frac{\partial H_2}{\partial a}(t,a^*_t)=0$.  On the other hand,  when ${\theta}\geqslant {\gamma e^{r(T-t)}} (\mathbb E[Y]+\alpha\VaR_{\alpha}(Y)-\alpha\ES_{\alpha}(Y))$, we have $\frac{\partial H_2}{\partial a}(t,a)\geqslant 0$ for all $a\in[0,\VaR_{\alpha}(Y)]$, and thus $a^*_t=0$ is the minimizer of \eqref{inf}. We complete the proof. 
\end{proof}

\begin{proof}[Proof of Proposition \ref{prop:5}]
By Theorem \ref{form}, we can divide $[0,\infty)$ into $[0,\VaR_{\alpha}(Y))$ and $[\VaR_{\alpha}(Y),\infty)$. We have
\begin{align*}
	I^*(t,y)=\min\left\{\max\left\{y+\frac{\lambda\alpha-(1+\theta)}{\alpha\gamma e^{r(T-t)}},I^*(t,\VaR_{\alpha}(Y))\right\},I^*(t,\VaR_{\alpha}(Y))+y-\VaR_{\alpha}(Y) \right\} 	
\end{align*} for $y\in[\VaR_{\alpha},\infty)$, and 
\begin{align*}
	I^*(t,y)=\min\{\max\{y+\frac{\lambda}{\gamma e^{r(T-t)}},I^*(t,\VaR_{\alpha}(Y))+y-\VaR_{\alpha}(Y),0\},y,I^*(t,\VaR_{\alpha}(Y))\}
\end{align*} for $y\in[0,\VaR_{\alpha}(Y))$. 

To determine the optimal indemnity function $I^*(t,y)$, only two parameters are needed:  $a=I^*(t,\VaR_{\alpha}(Y))\in [0,\VaR_{\alpha}(Y)]$ and $\lambda\in \mathbb R$. According to the values of $a$ and $\lambda$, $I^*(t,y)$ can take up to eight specific forms, which can be classified into two cases depending on the value of  $a$:
$$(1) \frac{1+\theta}{\alpha\gamma e^{r(T-t)}}+a-\VaR_{\alpha}(Y)\leqslant 0~~\text{and}~~  (2)\frac{1+\theta}{\alpha\gamma e^{r(T-t)}}+a-\VaR_{\alpha}(Y) > 0.$$ For the first case, where  $\frac{1+\theta}{\alpha\gamma e^{r(T-t)}}+a-\VaR_{\alpha}(Y)\leqslant 0$, the relationship between $I^*$ and  the parameters $a$, $\lambda$ is shown in Figure \ref{fig:es_1}. 
\begin{figure}[!htbp]
\centering
\includegraphics[scale=0.55]{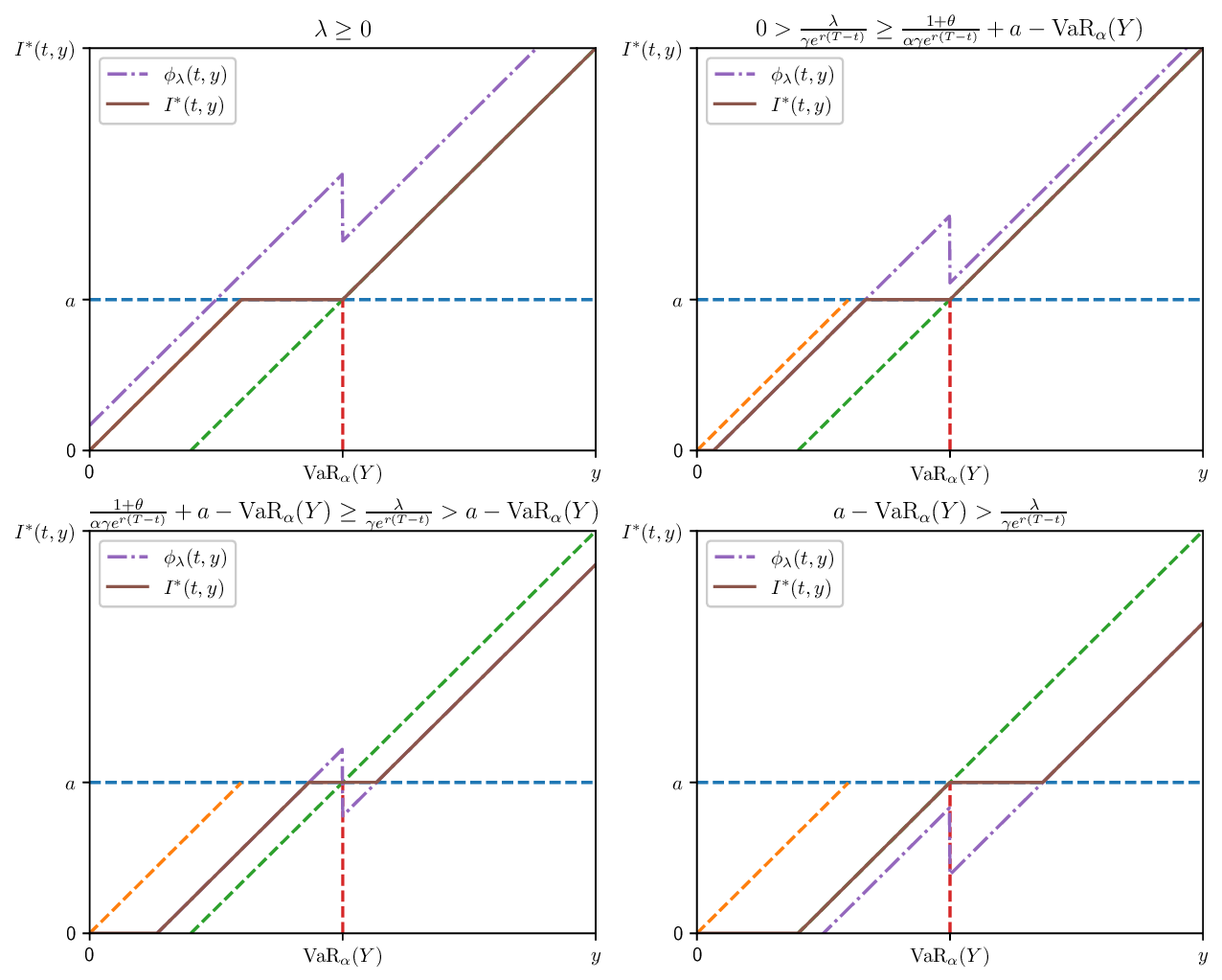}
\caption{The relationship between $I^*$ and $a$, $\lambda$ when $\frac{1+\theta}{\alpha\gamma e^{r(T-t)}}+a-\VaR_{\alpha}(Y)\leqslant 0$}
\label{fig:es_1}
\end{figure}
Next, we  show that if $I^*(t,y)$  is an equilibrium strategy in this case, then the corresponding value of $\lambda$ must be non-negative. In fact, we have 
\begin{align*}
	\mathbb E^{\mathbb Q}[I(t,Y)]&=\int_0^{\infty}I(t,y)\d F^{\mathbb Q}(y)=\dfrac{1}{\alpha}\int_{\VaR_{\alpha}(Y)}^{\infty}I(t,y)\d F(y)\\
	&=\dfrac{1}{\alpha}\int_0^{\alpha}I(t,\VaR_s(Y))\d s=\dfrac{1}{\alpha}\int_0^{\alpha}\VaR_s(I(t,Y))\d s\\
	&=\mathrm{ES}_{\alpha}(I(t,Y)).
\end{align*}
Thus, $\mathbb E^{\mathbb Q}(I(t,Y))$ only depends on the values of $I(t,y)$ for $y\in [\VaR_{\alpha}(Y),\infty)$. If $\lambda <0$, we define a new indemnity function
\begin{align*}
	\widetilde I(t,y)=\left\{
	\begin{aligned}
		&y\wedge a, &y\in &[0,\VaR_{\alpha}(Y)),\\
		&I^*(t,y), &y\in &[\VaR_{\alpha}(Y),\infty).
	\end{aligned}
	\right.
\end{align*}
Clearly,  $\widetilde I(t,y)>I^*(t,y)$ for $y\in[0,\VaR_{\alpha}(Y))$. Further,  we have
\begin{align*}
	&c(\widetilde I(t,\cdot))+\mathbb E[Y-\widetilde I(t,Y)]+\dfrac{\gamma e^{r(T-t)}}{2}\mathbb E[(Y-\widetilde I(t,Y))^2]\\
	=&c(I^*(t,\cdot))+\mathbb E[Y-\widetilde I(t,Y)]+\dfrac{\gamma e^{r(T-t)}}{2}\mathbb E[(Y-\widetilde I(t,Y))^2]\\
	<&c(I^*(t,\cdot))+\mathbb E[Y-I^*(t,Y)]+\dfrac{\gamma e^{r(T-t)}}{2}\mathbb E[(Y-I^*(t,Y))^2],
\end{align*}
which contradicts the assumption that $I^*$ is an equilibrium strategy. Therefore, in this case, the equilibrium indemnity function must have the following form
\begin{align}\label{es1}
	I^*(t,y)=y\wedge a+(y-\VaR_{\alpha}(Y))_+.
\end{align}
For the second case, where $\frac{1+\theta}{\alpha\gamma e^{r(T-t)}}+a-\VaR_{\alpha}(Y) > 0$, the relationship between $I^*$ and the parameters $a$, $\lambda$ is shown in Figure \ref{fig:es_2}.
\begin{figure}[!htbp]
\centering
\includegraphics[scale=0.55]{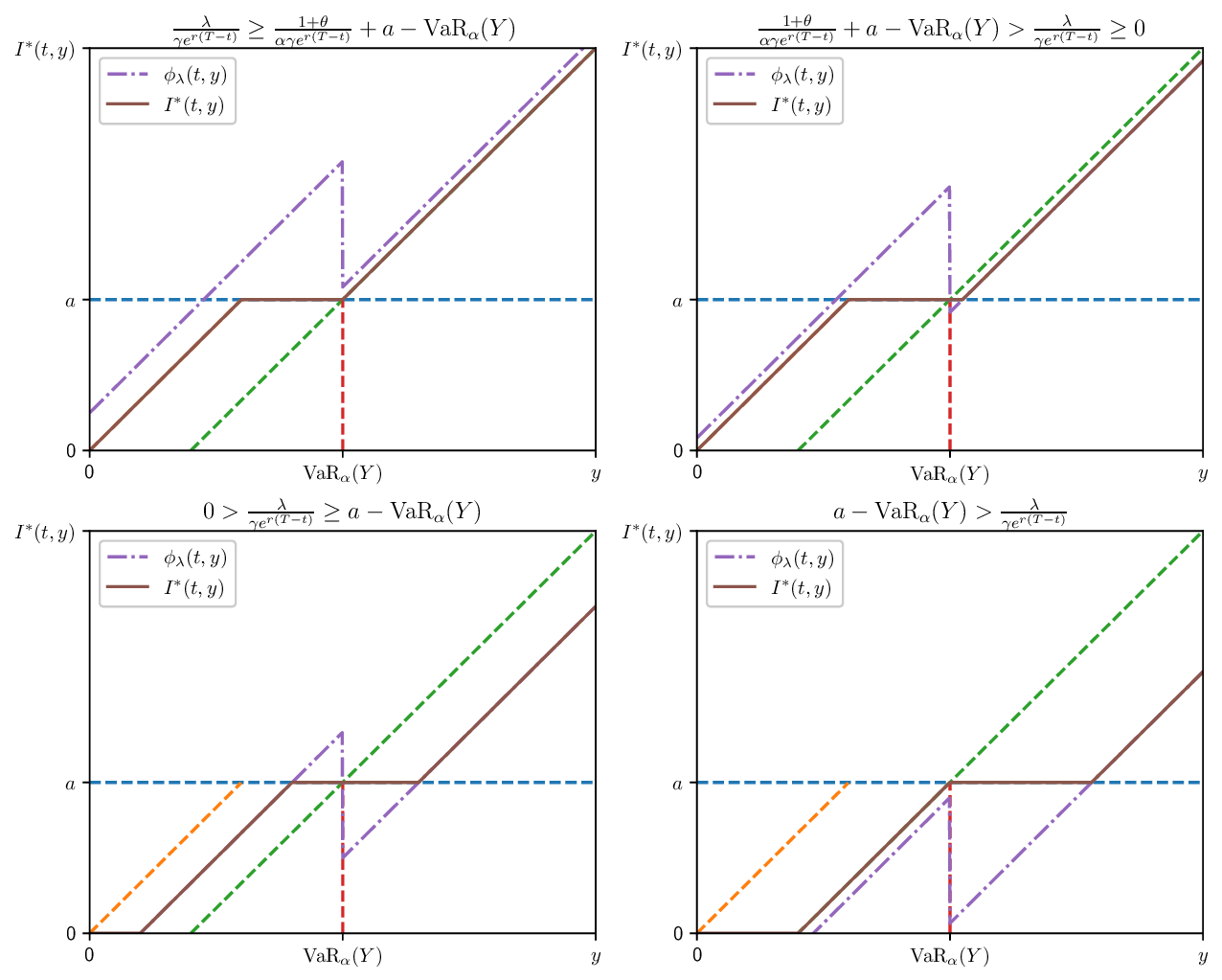}
\caption{The relationship between $I^*$ and $a$, $\lambda$ when $\frac{1+\theta}{\alpha\gamma e^{r(T-t)}}+a-\VaR_{\alpha}(Y)> 0$}
\label{fig:es_2}
\end{figure}
As in the first case, we can conclude that if $I^*(t,y)$  is an equilibrium strategy, then the corresponding $\lambda$ must be non-negative. Further, when $\lambda\geqslant \frac{1+\theta}{\alpha}+\gamma e^{r(T-t)}(a-\VaR_{\alpha}(Y))$, the optimal indemnity function $I^*$ takes the same form as in the first case. Therefore, we only need to consider $0\leqslant\lambda\leqslant \frac{1+\theta}{\alpha}+\gamma e^{r(T-t)}(a-\VaR_{\alpha}(Y))$, and the equilibrium indemnity function takes the following form
\begin{align}\label{es2}
I^*(t,y)=y\wedge a+(y-b)_+,\ \ b \in\left[\VaR_{\alpha}(Y),	a+\dfrac{1+\theta}{\alpha\gamma e^{r(T-t)}}\right].
\end{align}
Combining \eqref{es1} and \eqref{es2},  we get the desired result. 
\end{proof}

\begin{proof}[The proof of Proposition \ref{Prop:6}]
(i) When $\theta_1\geq\theta_2$, we have  LR$'(y)<0$. By Corollary \ref{coro:1},  the equilibrium strategy is  the form of
$ I^*(t,y)=(y-d)_+$
for some $d\geqslant 0$.  To determine $d$, substituting $I^*$ back into \eqref{inf} yields \begin{align*} &	c(I^*(t,\cdot))+\mathbb E[Y-I^*(t,Y)]+\dfrac{\gamma e^{r(T-t)}}{2}\mathbb E[(Y-I^*(t,Y))^2]\\ =&~(1+\theta)\int_d^{\infty}S^{\mathbb Q}(y)\d y+\int_0^dy\d F(y)+dS(d)+\dfrac{\gamma e^{r(T-t)}}{2}\left(\int_0^dy^2\d F(y)+d^2S(d)\right)\\	=&:H_3(t,d). \end{align*} 
Thus, by \eqref{LR:exp}, we have 
\begin{align*}
	\frac{\partial H_3}{\partial d}(t,d)=e^{-\frac{d}{\theta_1}}\left(1+\gamma e^{r(T-t)}d-(1+\theta)e^{-(\frac{1}{\theta_2}-\frac{1}{\theta_1})d}\right).
\end{align*}
By the implicit function theorem, there exists a unique continuous $d=d_t$ such that $\frac{\partial H_3}{\partial d}(t,d_t)=0$. Additionally, we observe that  $\frac{\partial H_3}{\partial d}(t,d)<0$ on $[0,d_t]$ and $\frac{\partial H_3}{\partial d}(t,d)>0$ on $[d_t,\infty)$. Thus, the optimal $d^*_t$ is given by \eqref{eq:d_ex1}.

(ii) When $\theta_1\leq\theta_2-\gamma e^{r(T-t)}\frac{\theta_2^2} {1+\theta}$,  we have  $\gamma e^{r(T-t)}-(1+\theta)\text{LR}'(y)\leqslant 0$ on $[0,\infty)$. By Theorem \ref{form},  the equilibrium strategy is $I^*(t,y)=x\wedge d$,
for some $d\geqslant0$. Substituting $I^*$ into \eqref{inf}, we have
\begin{align*}
	&c(I^*(t,\cdot))+\mathbb E[Y-I^*(t,Y)]+\dfrac{\gamma e^{r(T-t)}}{2}\mathbb E[(Y-I^*(t,Y))^2]\\
	=&(1+\theta)\mathbb E^{\mathbb Q}[Y\wedge d]+\mathbb E[(Y-d)_+]+\dfrac{\gamma e^{r(T-t)}}{2}[(Y-d)_+^2]\\
	=&(1+\theta)\left(\int_0^dy\d F^{\mathbb Q}(y)+dS^{\mathbb Q}(d)\right)+\int_d^{\infty}(y-d)\d F(y)+\dfrac{\gamma e^{r(T-t)}}{2}\int_d^{\infty}(y-d)^2\d F(y)\\
	=&(1+\theta)\left(\int_0^dy\d F^{\mathbb Q}(y)+dS^{\mathbb Q}(d)\right)+\int_d^{\infty}S(y)\d y+\gamma e^{r(T-t)}\int_d^{\infty}(y-d)S(y)\d y\\
	=&:H_4(t,d).
\end{align*}
Taking the derivative of $H_4(t,d)$, we get 
\begin{align*}
	\frac{\partial H_4}{\partial d}(t,d)=&~(1+\theta)S^{\mathbb Q}(d)-S(d)-\gamma e^{r(T-t)}\int_d^{\infty}S(y)\d y\\
	=&~(1+\theta)e^{-\frac{d}{\theta_2}}-e^{-\frac{d}{\theta_1}}-\gamma e^{r(T-t)}\int_d^{\infty}e^{-\frac{d}{\theta_1}}\d y\\
	=&~(1+\theta)e^{-\frac{d}{\theta_2}}-(1+\gamma e^{r(T-t)}\theta_1)e^{-\frac{d}{\theta_1}}.
\end{align*}
If $\theta\geq \gamma e^{r(T-t)}\theta_1$, then $\frac{\partial H_4}{\partial d}(t,d)\geqslant 0$  for all $d\geqslant 0$ and thus $d^*_t=0$ is the minimizer of $H_4(t,d)$. Otherwise, if $\theta<\gamma e^{r(T-t)}\theta_1< 1$, there exists $d^*_t>0$ such that $$\frac{1+\theta}{1+\gamma e^{r(T-t)}\theta_1}e^{(\frac{1}{\theta_1}-\frac{1}{\theta_2})d^*_t}=1.$$ For $d\in [0,d^*_t)$, we have  $\frac{\partial H_4}{\partial d}(t,d)<0$, and  for  $d\in[d^*_t,\infty)$, we have   $\frac{\partial H_4}{\partial d}(t,d)>0$. Thus, $$d^*_t=\frac{\theta_1\theta_2}{\theta_2-\theta_1}\ln\frac{1+\gamma e^{r(T-t)}\theta_1}{1+\theta}$$ is the minimizer of $H_4(t,d)$.

 (iii) When $\theta_2-\gamma e^{r(T-t)}\frac{\theta_2^2} {1+\theta} < \theta_1<\theta_2$, then there exists a $y_1>0$ such that $$\gamma e^{r(T-t)}-(1+\theta)\text{LR}'(y_1)=0.$$ By direct calculation, we know that $$y_1=\frac{\theta_1\theta_2}{\theta_2-\theta_1}r(T-t)+\frac{\theta_1\theta_2}{\theta_2-\theta_1}\ln\frac{\gamma\theta_2^2}{(1+\theta)(\theta_2-\theta_1)}.$$ On $(y_1,\infty)$, $\gamma e^{r(T-t)}-(1+\theta)\text{LR}'(y)<0$. By Theorem \ref{form},  the form of equilibrium strategy is 
\begin{align}\label{e1}
	I^*(t,y)=I^*(t,y_1)+y\wedge d-y_1,
\end{align}
for some $d\geqslant y_1$. On $[0,y_1]$, $\gamma e^{r(T-t)}-(1+\theta)\text{LR}'(y)\geqslant 0$ and by Theorem \ref{form}, we have
\begin{align}\label{e2}
	I^*(t,y)=\min\{\max\{\phi_{\lambda}(t,y),I^*(t,y_1)+y-y_1,0\},y,I^*(t,y_1)\}.
\end{align}
By combining \eqref{e1} and \eqref{e2}, we can express the equilibrium strategy $I^*(t,y)$ on $[0,\infty)$ as follows: 
$$ 
I^*(t,y)=\min\{\max\{\phi_{\lambda}\chi_{\{y\leqslant y_1\}},a+y-y_1,0\},y,a\}+(y-y_1)_+-(y-d)_+,	
$$ 
where $a\in[0,y_1]$, $d\in[y_1,\infty)$, and $\lambda\in\mathbb R$. Here,  $\phi_{\lambda}(t,y)$ is given by $$\phi_{\lambda}(t,y)=y-\frac{(1+\theta)\theta_1}{\theta_2\gamma e^{r(T-t)}}e^{(\frac{1}{\theta_1}-\frac{1}{\theta_2})y}+\frac{\lambda}{\gamma e^{r(T-t)}}.$$  The equilibrium strategy  is determined by finding the values of $(a,d,\lambda)$ that minimize the functional $H(t,I^*)$ in \eqref{inf}  which is feasible because it is continuous  with respect to  $(a,d,\lambda)$.
\end{proof}

\end{document}